\documentclass[12pt, reqno]{amsart}
\usepackage{amsmath, amsthm, amscd, amsfonts, amssymb, graphicx, color, mathrsfs,enumitem}
 \usepackage{multirow}
\usepackage[all]{xy}
\usepackage{slashed}
\usepackage{mathabx}
\usepackage{tkz-euclide}
\usetikzlibrary{patterns}

\usepackage{graphicx}

\usepackage{tipa}
\usepackage{soul}
\usepackage{cancel}
\usepackage{ulem}
\usepackage{mathtools}

\newcommand{\M}{\mathcal{M}}

\usepackage[
  textwidth=16cm,
  margin=2.5cm
]{geometry}
\usepackage{xcolor}

\newtheorem{theorem}{Theorem}[section]
\newtheorem{lemma}[theorem]{Lemma}

\newtheorem{proposition}[theorem]{Proposition}

\theoremstyle{definition}
\newtheorem{definition}[theorem]{Definition}

\theoremstyle{remark}
\newtheorem{remark}[theorem]{Remark}  
\numberwithin{equation}{section}

\newcommand{\Rn}{{\mathbb{R}}^{n}}

\newcommand{\APQ}{\mathcal{H}_{P, Q}}

\usepackage[
    colorlinks=true,
    citecolor=blue,
    linkcolor=blue,
    urlcolor=blue,
    backref=page
]{hyperref}

\begin{document}
\setcounter{page}{1}

\title[Semigroups Associated with Polynomial Anharmonic Oscillators]{Sharp Time-Decay Estimates for Fractional Heat Semigroups Associated with Polynomial Anharmonic Oscillators}

  \author[J. Delgado]{Julio Delgado}
\address{
  Julio Delgado:
  \endgraf
  Departmento de Matematicas
  \endgraf
  Universidad del Valle
  \endgraf
  Cali-Colombia
  \endgraf
    {\it E-mail address} {\rm delgado.julio@correounivalle.edu.co}}

\author[V. Kumar]{Vishvesh Kumar}
\address{
  Vishvesh Kumar:
  \endgraf
  Department of Mathematical Sciences 
  \endgraf
  Indian Institute of Technology (BHU) Varanasi, Varanasi, India
  \endgraf
  {\it E-mail address} {\rm  vishvesh.mat@iitbhu.ac.in}
  }

  \author[Shyam S. Mondal]{Shyam Swarup Mondal} \address{Shyam Swarup Mondal    \endgraf  Stat-Math Unit \endgraf Indian Statistical Institute (ISI) Kolkata, 203 BT Road,	Kolkata 700108, India} \email{mondalshyam055@gmail.com}


\date{\today}

\begin{abstract}
We investigate fractional heat semigroups generated by a class of anharmonic oscillators on $\mathbb R^n$ of the form $\mathcal H_{P,Q}=Q(D)+P(x),$ 
where $P\in\mathcal P_{2k}$ and $Q\in\mathcal P_{2\ell}$ are real-valued polynomials with anisotropic growth. Using the Weyl--H\"ormander calculus associated with the natural metric determined by $(P,Q)$, we show that the fractional powers $\mathcal H_{P,Q}^s$, $s>0$, are pseudo-differential operators with symbols in adapted classes $\Sigma_{P,Q}^{2s}$.

We prove fixed-time decay estimates for the fractional anharmonic heat semigroup $e^{-t\mathcal H_{P,Q}^s}$ on both Lebesgue and modulation spaces. In the Lebesgue setting, we establish sharp $L^p$--$L^q$ estimates for the full range $1\le p,q\le\infty$. For large time, the decay is exponential and governed by the smallest eigenvalue $\lambda_0$ of $\mathcal H_{P,Q}$, namely through the factor $e^{-t\lambda_0^s}$, while for small time the estimates reveal two distinct phase-space scales associated with the coercive growth of $P$ and $Q$, leading to anisotropic $L^p$--$L^q$ smoothing.

As applications, we study nonlinear fractional heat equations associated with $\mathcal H_{P,Q}^s$. We prove local well-posedness in the supercritical Lebesgue range $ p>\frac{n(\beta-1)}{2\ell s},$
derive a lower blow-up rate for finite-time blow-up solutions, and obtain critical small-data global existence. We further prove global well-posedness and exponential decay for small initial data in modulation spaces. These results extend the heat semigroup theory for harmonic and model anharmonic oscillators to a broad class of anisotropic polynomial Hamiltonians.
\end{abstract}
 \maketitle
\allowdisplaybreaks

\section{Introduction}

The purpose of this article is to study heat semigroups generated by fractional powers of a broad class of anharmonic oscillators on $\mathbb{R}^n$. More precisely, we consider Hamiltonians of the form
\[
\mathcal{H}_{P,Q}:=Q(D)+P(x),
\]
where $Q(\xi)$ and $P(x)$ are suitable real-valued polynomials. Throughout the paper, if $m\geq 1$ is an integer, we denote by $\mathcal{P}_{2m}$ the class of real-valued polynomials $P$ on $\Rn$ satisfying
\[
\liminf_{|x|\rightarrow\infty}\frac{P(x)}{|x|^{2m}}>0.
\]
Thus, we shall work with operators
\[
\mathcal{H}_{P,Q}=Q(D)+P(x),
\qquad
Q\in \mathcal{P}_{2\ell},\quad P\in \mathcal{P}_{2k},
\qquad
k,\ell\geq 1.
\]
Since the polynomials $P$ and $Q$ are bounded from below, the operator $\mathcal H_{P,Q}$ is bounded from below. Whenever necessary, one may add a sufficiently large constant $r_0\geq0$ in order to obtain a strictly positive self-adjoint realization. This harmless shift does not affect the symbolic or mapping properties under consideration; it only changes the bottom of the spectrum. Therefore, for notational simplicity, we write $\mathcal H_{P,Q}^s$ for the fractional powers of the positive realization of the anharmonic oscillator.

The polynomial structure of $P$ and $Q$ naturally gives rise to a Weyl--H\"ormander metric on phase space. Indeed, for suitable constants $p_0,q_0>0$, one associates with the pair $(P,Q)$ the metric
\[
g^{(P,Q)}
=
\frac{dx^2}{(p_0+q_0+P(x)+Q(\xi))^{1/k}}
+
\frac{d\xi^2}{(p_0+q_0+P(x)+Q(\xi))^{1/\ell}},
\]
which was shown in \cite{anh:cdr} to be a H\"ormander metric. The corresponding symbol classes may be described intrinsically as follows: for $m\in\mathbb{R}$, we say that $a\in \Sigma_{P,Q}^m$
if, for all multi-indices $\alpha,\beta$, there exists a constant $C_{\alpha,\beta}>0$ such that
\[
|\partial_x^\beta\partial_\xi^\alpha a(x,\xi)|
\leq
C_{\alpha,\beta}
(q+P(x)+Q(\xi))^{\frac m2-\frac{|\beta|}{2k}-\frac{|\alpha|}{2\ell}},
\]
for a sufficiently large $q>0$. This class is precisely adapted to the anharmonic oscillator $\mathcal{H}_{P,Q}$ and to its fractional powers.  Among the most extensively studied examples is the classical harmonic oscillator and its higher-order anharmonic generalizations of the form $H=(-\Delta)^{\ell}+|x|^{2k}.$
Such operators, usually referred to as anharmonic oscillators, arise naturally as Hamiltonians of quantum particles confined by polynomial potentials and also appear in the study of molecular vibrational dynamics \cite{Folland}. From the mathematical point of view, the spectral asymptotics of these operators were analyzed by B. Helffer and D. Robert in \cite{hr:sap3, hr:anosc2, hr:anosc}. Hamiltonians with polynomial potentials have since attracted considerable attention in spectral theory, microlocal analysis, and mathematical physics. More recently, one-dimensional anharmonic oscillators with nonsmooth potentials have also been investigated, revealing new phenomena and analytical difficulties that do not appear in the classical smooth setting \cite{sikora:a1}. Moreover, the spectral properties of anharmonic oscillators on $\mathbb{R}^n$ were studied in \cite{anh:cdr,Rober}. A significant development in this direction was obtained in \cite{anh:cdr2}, where a broad class of operators of the form $Q(D)+P(x),$
with smooth symbols $Q(\xi)$ and potentials $P(x)$ satisfying suitable growth assumptions, was studied within the framework of the Weyl--H\"ormander calculus. For further developments and recent results concerning anharmonic oscillators, we refer to \cite{CK,CK23,Carmona1}.

As discussed above, the study of heat semigroups generated by such operators plays a fundamental role in modern analysis and mathematical physics. These semigroups provide an indispensable framework for understanding regularity, smoothing effects, spectral properties, decay estimates, and the long-time behavior of solutions to linear and nonlinear evolution equations. Among the most classical examples is the heat semigroup $e^{-t\Delta}$ generated by the Laplacian, which has been extensively studied in partial differential equations, harmonic analysis, and mathematical physics \cite{Grafa}. In contrast, heat semigroups generated by harmonic and anharmonic oscillators exhibit substantially different behavior because of the presence of a confining potential. In particular, the spectrum becomes discrete and the large-time decay is exponential, rather than purely polynomial as in the case of the free heat semigroup.

Modulation spaces were introduced by H. Feichtinger in 1983; see \cite{feich:mod}. Since then, they have been extensively developed and have become an important tool in time-frequency analysis. The theory was later extended to the full quasi-Banach range $0<p,q\leq \infty$ in \cite{G}. For a historical account of the development of modulation spaces and a detailed overview of the literature, we refer the reader to Feichtinger's survey \cite{feich:hist}. Modulation spaces have found numerous applications in the study of linear and nonlinear partial differential equations, particularly in problems where phase-space localization and time-frequency methods play a central role. We refer to \cite{RSW12} for a survey of applications to nonlinear evolution equations, and to \cite{cp:cp1} for recent developments concerning nonlinear PDEs with initial data in modulation spaces.

The phase space analysis of the Hermite semigroup and its fractional powers was carried out in \cite{than:hs}, where fixed-time estimates on modulation spaces were established for the semigroup $e^{-tH^\beta}$, $H=-\Delta+|x|^2$, $\beta>0$, together with applications to nonlinear global well-posedness. Subsequently, fixed-time $L^p$-$L^q$ estimates in Lebesgue spaces for the fractional harmonic oscillator were obtained in \cite{than:hs1}, where nonlinear fractional heat equations driven by $H^\beta$ were also studied. These works are particularly relevant to the present paper, since they show that even in the harmonic case the fractional heat propagator is not a Fourier multiplier, and therefore the usual Fourier-analytic arguments for the fractional heat semigroup $e^{-t(-\Delta)^\beta}$ cannot be directly applied.

The present work extends this line of investigation from the harmonic oscillator to a much broader class of anharmonic oscillators. In particular, instead of the isotropic quadratic Hamiltonian $-\Delta+|x|^2,$ we consider general polynomial Hamiltonians $\mathcal H_{P,Q}=Q(D)+P(x),
\,
Q\in\mathcal P_{2\ell}, \,\, P\in\mathcal P_{2k}.$
This class includes the model higher-order anharmonic oscillator $(-\Delta)^\ell+|x|^{2k},$
but also allows much more general polynomial symbols and potentials. In this setting, the natural symbolic framework is no longer the isotropic Shubin calculus associated with $|x|^2+|\xi|^2$, but rather the anisotropic Weyl--H\"ormander calculus determined by the polynomial pair $(P,Q)$.

 The modulation-space estimates obtained for the  fractional harmonic oscillator in \cite{than:hs} were extended in \cite{Cardona} to a class of anharmonic operators of the form $(-\Delta)^l+|x|^{2k}$
on $\mathbb{R}^n$ in the framework of modulation spaces. More recently, the authors in \cite{Dolai, Dolai2} broadened this direction by considering weighted modulation spaces and a larger class of operators of the form $A(D)+V(x)$ using the similar techniques as in \cite{than:hs1,  Cardona} and tools developed in \cite{ Cardona} for anisotropic setting.
{\it We emphasize that our assumptions are weaker than the strict homogeneity that those assumed in \cite{Dolai, Dolai2}}. Indeed, in their work, the anharmonic oscillator is defined through symbols and potentials $A(\xi)$ and $V(x)$ which are strictly positive homogeneous polynomials of degrees $2\ell$ and $2k$, respectively; hence
\[
A(\xi)\asymp |\xi|^{2\ell},
\qquad
V(x)\asymp |x|^{2k}.
\]
Here we only assume the coercivity conditions
\begin{align} \label{corce}
    \liminf_{|\xi|\to\infty}\frac{Q(\xi)}{|\xi|^{2\ell}}>0,
\qquad
\liminf_{|x|\to\infty}\frac{P(x)}{|x|^{2k}}>0.
\end{align}
{\it Thus, every strictly positive homogeneous polynomial of the corresponding degree satisfies our assumptions, whereas the converse is not true.} Our framework, therefore, includes the homogeneous models as special cases and also permits non-homogeneous polynomial symbols and potentials, including lower-order terms.

The present paper complements and extends these works by establishing mapping properties on modulation spaces and sharp $L^p$-$L^q$ time-decay estimates in Lebesgue spaces for fractional heat semigroups generated by general anharmonic oscillators $\mathcal H_{P,Q}^s$, and by applying these estimates to nonlinear fractional heat equations.  Although our work is motivated by \cite{than:hs, Cardona, than:hs1}, the treatment of general anisotropic polynomial Hamiltonians necessitates a significant extension of the analytical framework developed therein.

We first establish the mapping properties of the fractional anharmonic heat semigroup on modulation spaces. This is the natural starting point of our analysis, since the operator $\mathcal H_{P,Q}=Q(D)+P(x)$ is intrinsically described through the phase-space geometry associated with the polynomial pair $(P,Q)$. More precisely, the Weyl--H\"ormander metric $g^{(P,Q)}$ and the corresponding symbol classes $\Sigma_{P,Q}^m$ allow us to treat $\mathcal H_{P,Q}^s$ and $e^{-t\mathcal H_{P,Q}^s}$ within a global pseudo-differential calculus. The modulation-space estimates obtained below are then used as one of the main ingredients in the derivation of the Lebesgue-space $L^p$-$L^q$ estimates.

Our first main result is the following fixed-time estimate on modulation spaces.

\begin{theorem} \label{mainthmest}
Let $s>0$ and let $0<p_1,p_2,q_1,q_2\leq \infty$. Define
\[
\frac{1}{\widetilde p}
:=
\max\left\{
\frac{1}{p_2}-\frac{1}{p_1},0
\right\},
\qquad
\frac{1}{\widetilde q}
:=
\max\left\{
\frac{1}{q_2}-\frac{1}{q_1},0
\right\},
\]
and
\[
\sigma
:=
\frac{n}{2s}
\left(
\frac{1}{k\widetilde p}
+
\frac{1}{\ell\widetilde q}
\right).
\]
Then the fractional anharmonic heat semigroup $e^{-t\APQ^s}$ satisfies, for every $t>0$,
\begin{equation}
\label{mappinganharmonic}
\Vert e^{-t\APQ^{s}} f\Vert_{\M^{p_2,q_2}(\mathbb{R}^n)}
\leq
C(t)
\Vert f\Vert_{\M^{p_1,q_1}(\mathbb{R}^n)},
\end{equation}
where
\begin{equation}
\label{heatest}
C(t)
=
C'
\begin{cases}
t^{-\sigma}, & 0<t\leq 1,\\
e^{-t\lambda_0^s}, & t\geq 1,
\end{cases}
\end{equation}
for some constant $C'>0$. Here $\lambda_0$ denotes the smallest eigenvalue of the anharmonic oscillator $\APQ$.
\end{theorem}

Theorem \ref{mainthmest} shows that the fractional anharmonic heat semigroup has two distinct regimes. For small time, it exhibits phase-space smoothing governed by the anisotropic parameters $k$ and $\ell$, which encode the growth of the potential $P(x)$ and the symbol $Q(\xi)$, respectively. For large time, the decay is exponential and is determined by the bottom of the spectrum, namely by the factor $e^{-t\lambda_0^s}.$

The modulation-space estimate above extends the known phase-space estimates for the Hermite semigroup and its fractional powers to a substantially broader class of anharmonic oscillators. In particular, in the harmonic case $H=-\Delta+|x|^2$, the underlying phase-space geometry is isotropic and is described by the Shubin weight $1+|x|+|\xi|$. In contrast, for the general anharmonic oscillator $\mathcal H_{P,Q}=Q(D)+P(x)$, the natural weight is equivalent to $(q+P(x)+Q(\xi))^{1/2},$
which reflects the anisotropic growth of the polynomial Hamiltonian. Thus Theorem \ref{mainthmest} may be viewed as a phase-space extension of the fractional Hermite theory \cite{than:hs} to the Weyl--H\"ormander setting associated with general polynomial Hamiltonians.

Although $L^p$-$L^q$ estimates for semigroups have been extensively investigated in various settings over the last several decades, we briefly recall those most closely related to the present work. We begin with the fractional heat semigroup $e^{-t(-\Delta)^\beta}$; see \cite{MY, PV, We}. Its mapping properties are well understood, owing to its central role in the analysis of nonlocal partial differential equations and in numerous models arising in mathematical physics. Fixed-time decay estimates for heat semigroups associated with the Hermite operator have been obtained in \cite{than:hs1} in Lebesgue spaces. We also note that the $L^p$-$L^q$ boundedness of spectral multipliers and Fourier multipliers associated with anharmonic oscillators was studied in \cite{CK} and \cite{CK23}, respectively. Furthermore, $L^p$-$L^q$ boundedness results for pseudo-differential operators and Fourier multipliers in various frameworks can be found in \cite{CKRT20, CKR23, CDKR23, RT24, ANR19, AR19, KumarPA25, KR23} and the references therein. We also refer to \cite{TS, Bre, Nis, SRR,Jia} and the references therein for $L^p$-$L^q$ type estimates for wave equations.

The second main objective of this paper is to establish sharp $L^p$-$L^q$ time-decay estimates for the fractional heat semigroup $e^{-t\mathcal{H}_{P,Q}^{s}},
\, s>0,$
in the full range $1\leq p,q\leq\infty$. The last objective is to apply these estimates to nonlinear heat equations associated with $\mathcal{H}_{P,Q}^s$, both in Lebesgue spaces and in modulation spaces.


Among the most classical examples is the heat semigroup generated by the Laplacian. Its heat kernel is explicitly given by
\[
G_t(x,y)
=
(4\pi t)^{-\frac n2}
\exp\left(-\frac{|x-y|^2}{4t}\right),
\]
and this formula yields the standard estimate
\[
\Vert e^{t\Delta}f\Vert_{L^q(\mathbb{R}^n)}
\leq
C t^{-\frac n2(\frac1p-\frac1q)}
\Vert f\Vert_{L^p(\mathbb{R}^n)},
\qquad
1\leq p\leq q\leq\infty.
\]
For the fractional heat semigroup $e^{-t(-\Delta)^s}$, the corresponding estimate becomes
\[
\Vert e^{-t(-\Delta)^s}f\Vert_{L^q(\mathbb{R}^n)}
\leq
C t^{-\frac n{2s}(\frac1p-\frac1q)}
\Vert f\Vert_{L^p(\mathbb{R}^n)},
\]
whereas for the higher-order heat semigroup $e^{-t(-\Delta)^\ell}$ one obtains
\[
\Vert e^{-t(-\Delta)^\ell}f\Vert_{L^q(\mathbb{R}^n)}
\leq
C t^{-\frac n{2\ell}(\frac1p-\frac1q)}
\Vert f\Vert_{L^p(\mathbb{R}^n)}.
\]
In these homogeneous cases, the decay rate is dictated by the natural parabolic scaling of the underlying operator.

The situation is substantially different for harmonic and anharmonic oscillators. A basic model is the harmonic oscillator $H=-\Delta+|x|^2,$
whose heat kernel is explicitly described by the Mehler formula. This already reveals a fundamental difference from the free heat equation: the confining quadratic potential produces a discrete spectrum and exponential large-time decay. The fixed-time $L^p$-$L^q$ estimates for the fractional harmonic oscillator were established in \cite{than:hs1}, where nonlinear fractional heat equations driven by $H^\beta$ were also investigated. The present paper extends this line of investigation from the harmonic oscillator to a broad class of anharmonic oscillators. Instead of the isotropic quadratic Hamiltonian
$-\Delta+|x|^2,$ we consider general polynomial Hamiltonians
\[
Q(D)+P(x),
\qquad
Q\in\mathcal P_{2\ell},\quad P\in\mathcal P_{2k}.
\]
In particular, our results cover the higher-order anharmonic model $(-\Delta)^\ell+|x|^{2k},
\,
\ell,k\in\mathbb N,$
as well as much more general polynomial symbols and potentials and novel even in these particular setting. 

An important spectral feature of the operators considered here is that, after the harmless positivity shift described above, $\mathcal{H}_{P,Q}$ has purely discrete spectrum. We write its eigenvalues as
\[
0<\lambda_0\leq \lambda_1\leq \lambda_2\leq \cdots,
\qquad
\lambda_j\rightarrow\infty.
\]
The eigenfunctions form an orthonormal basis of $L^2(\Rn)$ and belong to $\mathcal{S}(\Rn)$. Moreover, the spectral asymptotics obtained in \cite{anh:cdr} give
\[
\lambda_j\sim C_{k,\ell}j^{\frac{2k\ell}{n(k+\ell)}},
\qquad
j\rightarrow\infty.
\]
The smallest eigenvalue $\lambda_0$ plays a decisive role in the large-time behavior of the heat semigroup. Indeed, by the spectral theorem,
\[
e^{-t\mathcal{H}_{P,Q}^s}f
=
\sum_{j=0}^{\infty}
e^{-t\lambda_j^s}P_jf,
\]
and the factor $e^{-t\lambda_0^s}$ is the slowest exponential decay rate appearing in this expansion. Thus, the bottom of the spectrum determines the optimal large-time exponential decay of the semigroup. In the harmonic case $H=-\Delta+|x|^2$ on $\mathbb R^d$, the bottom of the spectrum is explicit and equals $d$, leading to the decay factor $e^{-td^\beta}$ for $e^{-tH^\beta}$. In the present general anharmonic setting, the bottom of the spectrum is not explicit in general, and the sharp large-time decay is governed by $e^{-t\lambda_0^s}$.\\

We now state our second and most significant result on $L^p$-$L^q$ time-decay.

\begin{theorem}
\label{Lpmapping}
Let $s>0$. Then the following estimates hold.

\begin{itemize}
\item If $t>1$, then, for all $p,q\in[1,\infty]$,
\begin{equation}
\label{lpest}
    \Vert e^{-t\APQ^s}f\Vert_{L^q(\mathbb R^n)}
    \leq
    C e^{-t\lambda_0^s}
    \Vert f\Vert_{L^p(\mathbb R^n)}.
\end{equation}

\item If $0<t\leq1$, then
\begin{equation}
\label{lpest1}
    \Vert e^{-t\APQ^s}f\Vert_{L^q(\mathbb R^n)}
    \leq
    C t^{-\sigma_s}
    \Vert f\Vert_{L^p(\mathbb R^n)},
\end{equation}
where $\sigma_s$ is given as follows:
\begin{itemize}
\item[(i)] If $p,q\in(1,\infty)$, then
\[
    \sigma_s=
    \begin{cases}
    \displaystyle
    \frac{n}{2k s}
    \left(\frac1q-\frac1p\right),
    & q\leq p,\\[3mm]
    \displaystyle
    \frac{n}{2\ell s}
    \left(\frac1p-\frac1q\right),
    & p\leq q.
    \end{cases}
\]

\item[(ii)] If $p=1$ and $q=\infty$, then $ \sigma_s=\frac{n}{2\ell s}.$

\item[(iii)] If $p=1$ and $2\leq q<\infty$, then $ \sigma_s=
    \frac{n}{2\ell s}
    \left(1-\frac1q\right).$

\item[(iv)] If $1<p<\infty$ and $q=1$, then $ \sigma_s=
    \frac{n}{2k s}
    \left(1-\frac1p\right).$
\end{itemize}
\end{itemize}
\end{theorem}

We emphasize that Theorem \ref{Lpmapping} appears to be new even for the model anharmonic oscillator $(-\Delta)^\ell+|x|^{2k},
\,\,
\ell,k\in\mathbb{N}.$
Although heat semigroups associated with the harmonic oscillator and with certain anharmonic operators have been studied previously \cite{than:hs1, CK, CK23}, sharp $L^p$-$L^q$ estimates for the fractional semigroup $e^{-t\left((-\Delta)^\ell+|x|^{2k}\right)^s}$
in the full range $1\leq p,q\leq\infty$ do not seem to have been available in this generality.  Thus, our result extends the known theory from the harmonic oscillator to higher-order anharmonic oscillators and from integer-order heat flows to fractional powers of such operators.  The large-time estimate \eqref{lpest} is sharp. Indeed, since $\lambda_0$ is the smallest eigenvalue, the term $e^{-t\lambda_0^s}$ is the slowest exponential decay rate in the spectral expansion. Hence, in general, no faster exponential decay can hold. This feature is in sharp contrast with the free heat semigroup generated by $(-\Delta)^s$ or $(-\Delta)^\ell$ on $\mathbb{R}^n$, where the spectrum is continuous and starts at zero, leading only to polynomial decay in time.

\begin{remark} The exponents in \eqref{lpest1} reflect the separate influence of the coercive growth rates of $P$ and $Q$. Under \eqref{corce}, the exponent $\frac{n}{2ks}$ is associated with the spatial growth of $P$, whereas $\frac{n}{2\ell s}$ is associated with the frequency growth of $Q$. Thus the short-time $L^p$--$L^q$ behaviour of the semigroup retains distinct signatures of the two symbols. In the model case $P(x)\sim |x|^{2k}, \,\, Q(\xi)\sim |\xi|^{2\ell},$ the parameter $k$ measures the strength of spatial confinement, while $\ell$ corresponds to the order of the differential part. For the harmonic oscillator ($k=\ell=1$), these two scales coincide, recovering the classical $L^p$--$L^q$ behaviour (see \cite{than:hs1}). \end{remark}

As an application of the above decay estimates, we study the nonlinear fractional heat equation
\begin{equation}
\label{eq6.1inro}
\begin{cases}
\partial_t u+\APQ^s u=|u|^{\beta-1}u,
\quad (t,x)\in (0,\infty)\times\mathbb{R}^n,\\
u(0,x)=u_0(x),
\end{cases}
\end{equation}
where $\beta>1$, $s>0$ and initial data $u_0$ is taken from the Lebesgue space.

This problem is closely related to the classical Fujita theory. For the standard heat equation
\[
\partial_t u-\Delta u=|u|^{\beta-1}u,
\]
 it was shown by Fujita \cite{Fujita1966} and Hayakawa \cite{Hayakawa1973}, that the critical exponent is $\beta_F=1+\frac{2}{n}.$
If $1<\beta\leq\beta_F$, nontrivial nonnegative solutions blow up in finite time, whereas for $\beta>\beta_F$, global solutions exist for sufficiently small initial data. For the higher-order heat equation
\[
\partial_t u+(-\Delta)^\ell u=|u|^{\beta-1}u,
\qquad
\ell\in\mathbb{N},
\]
the corresponding Fujita exponent is
 $\beta_F=1+\frac{2\ell}{n}.$
Moreover, the scaling-critical Lebesgue exponent is $p_c=\frac{n(\beta-1)}{2\ell}.$
This follows from the scaling invariance
\[
u_\lambda(t,x)
=
\lambda^{\frac{2\ell}{\beta-1}}
u(\lambda^{2\ell}t,\lambda x).
\] We refer to excellent monograph \cite{GMP} for more detailed exposition.
Similarly, for the fractional heat equation
\[
\partial_t u+(-\Delta)^s u=|u|^{\beta-1}u,
\] the Fujita exponent is $\beta_F=1+\frac{2s}{n},$
and the scaling-critical Lebesgue exponent is $p_c=\frac{n(\beta-1)}{2s}.$
We refer to \cite{We,Mio,Ha,Br,MiaoL,KumarTorebek} and the references therein for wellposedness results on (fractional) heat equations.

In the anharmonic setting considered here, the situation is fundamentally different. The operator $\mathcal{H}_{P,Q}=Q(D)+P(x)$
contains both a differential part and a confining polynomial potential. The potential term destroys the dilation symmetry enjoyed by the homogeneous operators $(-\Delta)^\ell$ and $(-\Delta)^s$. Consequently, the nonlinear equation
\[
\partial_t u+\mathcal{H}_{P,Q}^s u=|u|^{\beta-1}u
\]
is not scaling invariant. Therefore, the critical exponent cannot be obtained from a scaling argument. Instead, it is determined by the short-time $L^p$-$L^q$ estimates for the semigroup $e^{-t\mathcal{H}_{P,Q}^s}$.

The decay estimates in Theorem \ref{Lpmapping} show that the effective diffusion parameter in the smoothing regime is governed by $\ell$, which is determined by the polynomial geometry of $P$ and $Q$. This leads naturally to the critical Lebesgue exponent
\[
p_c^s
=
\frac{n(\beta-1)}{2\ell s}.
\]
Thus, $p_c^s$ plays for the nonlinear anharmonic heat equation the same role that $\frac{n(\beta-1)}{2\ell}$
plays for the higher-order heat equation generated by $(-\Delta)^\ell$, and that $\frac{n(\beta-1)}{2s}$
plays for the fractional heat equation generated by $(-\Delta)^s$. Equivalently, from the viewpoint of $L^1$-based Fujita theory, the decay mechanism suggests the Fujita-type threshold
\[
\beta_F
=
1+\frac{2\ell s}{n}.
\]
However, unlike the classical and homogeneous higher-order models, this exponent is not a consequence of scaling invariance. It is instead extracted from the precise semigroup decay estimates associated with the non-homogeneous anharmonic operator.
We emphasize that the nonlinear results obtained here are also new even for the model operator $\left((-\Delta)^\ell+|x|^{2k}\right)^s.$

We first prove local well-posedness in the supercritical Lebesgue regime.

\begin{theorem}
\label{them6.1intro}
Let $1<p<\infty$, $s>0$, $\beta>1$, and $u_0\in L^p(\mathbb{R}^n)$. If
\[
p>p_c^s:=\frac{n(\beta-1)}{2\ell s},
\]
then there exists $T>0$ such that \eqref{eq6.1inro} has a solution $u\in C([0,T],L^p(\mathbb{R}^n)).$
Moreover, $u$ extends to a maximal interval $[0,T_{\max})$ such that either $T_{\max}=\infty$, or $T_{\max}<\infty$ and $\lim_{t\rightarrow T_{\max}}
\Vert u(t)\Vert_{L^p(\mathbb{R}^n)}
= \infty.$
\end{theorem}

We next derive a lower bound for the blow-up rate of finite-time blow-up solutions.

\begin{theorem}

Let $1<p<\infty$, $s>0$, $\beta>1$, and $u_0\in L^p(\mathbb R^n)$. Assume that
\[
    p>p_c^s:=\frac{n(\beta-1)}{2\ell s}.
\]
Let $u$ be the maximal solution of \eqref{eq6.1inro} obtained in Theorem \ref{them6.1intro}. If $T_{\max}<\infty$, then
\[
    \Vert u(t)\Vert_{L^p(\mathbb R^n)}
    \geq
    C
    (T_{\max}-t)^{
    \frac{n}{2p\ell s}  -
    \frac{1}{\beta-1}
    }
\]
for all $0\leq t<T_{\max}$.
\end{theorem}

We also obtain a global small-data result at the critical exponent.

\begin{theorem}
Let $1<p<\infty$, $s>0$, $\beta>1$, and $u_0\in L^p(\mathbb{R}^n)$. If
\[
p=p_c^s:=\frac{n(\beta-1)}{2\ell s}
\]
and $\Vert u_0\Vert_{L^{p_c^s}(\mathbb{R}^n)}$ is sufficiently small, then $T_{\max}=\infty.$
\end{theorem}

In the final part of the paper, we extend the nonlinear theory to modulation spaces. Modulation spaces are naturally suited to the analysis of pseudo-differential operators and time-frequency localization.  The use of modulation spaces in nonlinear parabolic problems goes back, in particular, to the work of Iwabuchi, who studied Navier--Stokes equations and nonlinear heat equations in modulation spaces with negative derivative indices and established local and global existence results for suitable initial data \cite{Iwabuchi}; see also the related discussion in \cite{RSW12}. In the fractional setting, Chen, Ding, Deng and Fan obtained estimates for fractional power dissipative equations in function spaces, including modulation-space frameworks \cite{CDDF12}. More recently, Bhimani \cite{BhimaniBlowup} investigated blow-up phenomena for nonlinear fractional heat equations
in modulation and Fourier amalgam spaces on both the torus and the Euclidean space, complementing known local and small-data global well-posedness results in modulation spaces. 

For heat equations associated with operators with confining potentials, Cordero \cite{CorderoHermite} studied the local well-posedness of nonlinear heat equations driven by fractional Hermite operators $H^\beta=(-\Delta+|x|^2)^\beta,
\quad 0<\beta\leq 1,$ in modulation spaces, using tools from microlocal and time-frequency analysis. In \cite{than:hs}, the authors studied the global well-posedness of the nonlinear fractional heat equations associated with the harmonic oscillator.  These works show that modulation spaces provide a natural framework for heat equations generated by operators that are not Fourier multipliers.

In the anharmonic setting, modulation-space estimates for heat semigroups associated with operators of the form $(-\Delta)^l+|x|^{2k}$
were obtained in \cite{Cardona}, and subsequent developments in \cite{Dolai, Dolai2} for  Hartree-type nonlinear heat equations associated with fractional anharmonic oscillators in weighted modulation spaces. The present paper continues this line of research by treating a broader class of fractional anharmonic oscillators $\mathcal H_{P,Q}^s=(Q(D)+P(x))^s$
within the Weyl--H\"ormander calculus associated with the polynomial pair $(P,Q)$. We use the resulting modulation-space decay estimates to prove global small-data well-posedness for nonlinear fractional heat equations with initial data in $\M^{p,q}$.


We consider the nonlinear fractional heat equation
\begin{equation}
\label{nonpro}
    \begin{cases}
        \partial_t u+\APQ^s u=\lambda |u|^{2\beta}u,\\
        u(0,x)=u_0(x),
    \end{cases}
\end{equation}
for $(t,x)\in(0,\infty)\times\mathbb R^n$, where $\beta\in\mathbb N$, $\lambda\in\mathbb C$, and $s>0$.

Combining the modulation-space decay estimates for $e^{-t\APQ^s}$ with a fixed-point argument, we prove the following global small-data result.

\begin{theorem}
\label{mod_global_nonlinear}
Let $p,q\in[1,\infty]$ be such that
$2\beta+1\leq q',
   \,
    \frac{\beta n}{s\ell}<q',$
where $q'$ denotes the conjugate exponent of $q$. Then there exists $\epsilon>0$ such that, if $\Vert u_0\Vert_{\mathcal M^{p,q}}\leq \epsilon,$
the problem \eqref{nonpro} admits a global mild solution $ u\in L^\infty([0,\infty),\mathcal M^{p,q}).$ Moreover, if $p<\infty$, then $u\in C([0,\infty),\mathcal M^{p,q}).$ Furthermore, if $\epsilon>0$ is chosen sufficiently small, then the solution decays exponentially in time. More precisely, $u\in Y,$
where
\begin{equation}
\label{spaceY}
    Y
    :=
    \left\{
    u\in L^\infty([0,\infty),\mathcal M^{p,q})
    :
    \Big\Vert
    e^{t\lambda_0^s}
    \Vert u(t,\cdot)\Vert_{\mathcal M^{p,q}}
    \Big\Vert_{L_t^\infty([0,\infty))}
    <\infty
    \right\},
\end{equation}
equipped with the norm
\begin{equation}
\label{Ynorm}
    \Vert u\Vert_Y
    :=
    \Big\Vert
    e^{t\lambda_0^s}
    \Vert u(t,\cdot)\Vert_{\mathcal M^{p,q}}
    \Big\Vert_{L_t^\infty([0,\infty))}.
\end{equation}
Here $\lambda_0$ denotes the smallest eigenvalue of the anharmonic oscillator $\APQ$.
\end{theorem}


The paper is organized as follows. In Section \ref{preliminaries}, we recall the necessary background on modulation spaces and Weyl--H\"ormander calculus, and introduce the symbol classes adapted to the polynomial pair $(P,Q)$. In Section \ref{Modul}, we establish the mapping properties of the fractional anharmonic heat semigroup on modulation spaces. In Section \ref{Lpprop}, we prove the $L^p$-$L^q$ estimates for $e^{-t\APQ^s}$ on Lebesgue spaces. Section \ref{application} is devoted to applications to nonlinear fractional heat equations in Lebesgue spaces, including local and global well-posedness results. Finally, in Section \ref{Modulation}, we prove the global small-data well-posedness result in modulation spaces.

\section{Preliminaries: Modulation spaces and Weyl--H\"ormander calculus}
\label{preliminaries}

In this section we recall the basic notions on modulation spaces and on the Weyl--H\"ormander pseudo-differential calculus that will be used throughout the paper. For the theory of modulation spaces we refer the reader to Gr\"ochenig's monograph \cite{groch:bK}, in particular Chapter 11. For a detailed account of the Weyl--H\"ormander calculus we refer to \cite[Section 18.5]{ho:apde2} and to \cite{ho:apde2,le:book,b-l:qua,Robert1,anh:cdr2}.

We begin with modulation spaces. A {\it modulation weight function} on $\mathbb R^{2n}$ is a non-negative locally integrable function. We use the adjective ``modulation'' in order to distinguish these weights from the weights associated with H\"ormander metrics, which will be introduced later.

A modulation weight $v$ on $\mathbb R^{2n}$ is called {\it submultiplicative} if
\begin{equation}
\label{polyw0}
    v(X+Y)\leq v(X)v(Y),\qquad X,Y\in\mathbb R^{2n}.
\end{equation}
Given such a weight $v$, another modulation weight $w$ is said to be { \it $v$-moderate} if
\begin{equation}
    w(X+Y)\leq v(X)w(Y),\qquad X,Y\in\mathbb R^{2n}.
\end{equation}
An important family of weights is given by the polynomial weights
\begin{equation}
\label{polyw}
    v_s(x,\xi)=\big(1+|x|^2+|\xi|^2\big)^{s/2}.
\end{equation}
Weights that are $v_s$-moderate for some $s$ are called {\it polynomially moderate}.

Let $w$ be a modulation weight on $\mathbb R^{2n}$, let $0<p,q\leq \infty$, and let $g\in \mathcal S(\mathbb R^n)\setminus\{0\}$ be a window function. The short-time Fourier transform of $f\in\mathcal S'(\mathbb R^n)$ with respect to $g$ is defined by
\begin{equation}
\label{stft1}
    V_gf(x,\xi)
    =
    \int_{\mathbb R^n}f(y)\overline{g(y-x)}e^{-iy\cdot\xi}\,dy .
\end{equation}
The modulation space $\mathcal M_w^{p,q}(\mathbb R^n)$ consists of all tempered distributions $f\in\mathcal S'(\mathbb R^n)$ such that
\begin{equation}
\label{EQ:modul}
\Vert f\Vert_{\mathcal M_w^{p,q}}
:=
\Vert V_gf\Vert_{L_w^{p,q}}
:=
\left(
\int_{\mathbb R^n}
\left(
\int_{\mathbb R^n}
|V_gf(x,\xi)|^p w(x,\xi)^p\,dx
\right)^{q/p}
d\xi
\right)^{1/q}
<\infty ,
\end{equation}
with the usual modifications when $p=\infty$ or $q=\infty$.

The space $\mathcal M_w^{p,q}(\mathbb R^n)$, equipped with the above quasi-norm, is a quasi-Banach space and is independent of the particular choice of non-zero window $g$. If $p,q\geq 1$, then $\mathcal M_w^{p,q}(\mathbb R^n)$ is a Banach space. In this Banach range, its dual is given by
\[
    \big(\mathcal M_w^{p,q}(\mathbb R^n)\big)'
    =
    \mathcal M_{1/w}^{p',q'}(\mathbb R^n),
\]
where $p'$ and $q'$ are the conjugate exponents of $p$ and $q$, respectively. Moreover, if $p_1\leq p_2,\quad q_1\leq q_2,\quad w_1\geq w_2,$
then the continuous embedding $  \mathcal M_{w_1}^{p_1,q_1}(\mathbb R^n)
    \hookrightarrow
    \mathcal M_{w_2}^{p_2,q_2}(\mathbb R^n)$
holds.

We next recall the basic elements of the Weyl--H\"ormander calculus. For a symbol
$a\in\mathcal S'(\mathbb R^n\times\mathbb R^n)$, its Weyl quantization is defined by
\[
    a^w(x,D)u(x)
    =
    \frac{1}{(2\pi)^n}
    \int_{\mathbb R^n}\int_{\mathbb R^n}
    e^{i\langle x-y,\xi\rangle}
    a\left(\frac{x+y}{2},\xi\right)
    u(y)\,dy\,d\xi ,
    \qquad u\in\mathcal S(\mathbb R^n).
\]
More generally, for $t\in\mathbb R$, the $t$-quantization of $a$ is defined by
\[
    a_t(x,D)u(x)
    =
    \frac{1}{(2\pi)^n}
    \int_{\mathbb R^n}\int_{\mathbb R^n}
    e^{i\langle x-y,\xi\rangle}
    a(tx+(1-t)y,\xi)
    u(y)\,dy\,d\xi .
\]
The Weyl quantization corresponds to the case $t=\frac12$, namely $a_{1/2}(x,D)=a^w(x,D),$
whereas the case $t=1$ gives the Kohn--Nirenberg quantization
\[
    a(x,D)u(x)
    =
    a_1(x,D)u(x)
    =
    \frac{1}{(2\pi)^n}
    \int_{\mathbb R^n}
    e^{i\langle x,\xi\rangle}
    a(x,\xi)\widehat u(\xi)\,d\xi .
\]

We now recall H\"ormander metrics, which will play a central role in the analysis of the anharmonic oscillators considered in this paper.

\begin{definition}[H\"ormander metric]
\label{HM}
Let $g_X$ be a positive definite quadratic form on $\mathbb R^{2n}$ for each $X\in\mathbb R^{2n}$. We say that $g=\{g_X\}_{X\in\mathbb R^{2n}}$ is a H\"ormander metric if the following conditions are satisfied.

\begin{enumerate}
\item[\rm I.] \textbf{Slowness.} There exists a constant $C>0$ such that
\[
    g_X(X-Y)\leq C^{-1}
    \quad\Longrightarrow\quad
    \left(\frac{g_X(T)}{g_Y(T)}\right)^{\pm1}\leq C
\]
for all $T\in\mathbb R^{2n}\setminus\{0\}$.

\item[\rm II.] \textbf{Uncertainty principle.} Let
\[
    \sigma(Y,Z)=\langle z,\eta\rangle-\langle y,\zeta\rangle,
    \qquad
    Y=(y,\eta),\quad Z=(z,\zeta),
\]
be the standard symplectic form on $\mathbb R^{2n}$. Define the dual metric by
\[
    g_X^\sigma(T)
    =
    \sup_{W\neq0}
    \frac{\sigma(T,W)^2}{g_X(W)}.
\]
The metric $g$ satisfies the uncertainty principle if
\[
    \lambda_g(X)
    :=
    \inf_{T\neq0}
    \left(
    \frac{g_X^\sigma(T)}{g_X(T)}
    \right)^{1/2}
    \geq 1 ,
    \qquad X\in\mathbb R^{2n}.
\]

\item[\rm III.] \textbf{Temperateness.} There exist constants $\overline C>0$ and $J\in\mathbb N$ such that
\[
    \left(
    \frac{g_X(T)}{g_Y(T)}
    \right)^{\pm1}
    \leq
    \overline C
    \big(1+g_Y^\sigma(X-Y)\big)^J
\]
for all $X,Y,T\in\mathbb R^{2n}$.
\end{enumerate}
\end{definition}

The function $\lambda_g$ is called the uncertainty parameter. The associated Planck function is defined by
\[
    h_g(X)^2
    =
    \sup_{T\neq0}
    \frac{g_X(T)}{g_X^\sigma(T)} .
\]
Thus,
\[
    h_g(X)=\lambda_g(X)^{-1}.
\]

\begin{definition}[$g$-weight]
\label{GW}
Let $M:\mathbb R^{2n}\to\mathbb R^+$ be a positive function.

\begin{itemize}
\item We say that $M$ is {\it $g$-continuous} if there exists $\widetilde C>0$ such that
\[
    g_X(X-Y)\leq \widetilde C^{-1}
    \quad\Longrightarrow\quad
    \left(
    \frac{M(X)}{M(Y)}
    \right)^{\pm1}
    \leq \widetilde C .
\]

\item We say that $M$ is {\it $g$-temperate} if there exist $\widetilde C>0$ and $N\in\mathbb N$ such that
\[
    \left(
    \frac{M(X)}{M(Y)}
    \right)^{\pm1}
    \leq
    \widetilde C
    \big(1+g_Y^\sigma(X-Y)\big)^N .
\]
\end{itemize}
If $M$ is both $g$-continuous and $g$-temperate, then $M$ is called a {\it $g$-weight}.
\end{definition}

Given a H\"ormander metric $g$ and a $g$-weight $M$, the associated symbol class $S(M,g)$ is defined as follows.

\begin{definition}
Let $g$ be a H\"ormander metric and let $M$ be a $g$-weight. The class $S(M,g)$ consists of all functions $\sigma\in C^\infty(\mathbb R^{2n})$ such that, for every integer $j\geq0$, there exists a constant $C_j>0$ satisfying
\begin{equation}
\label{inwhk}
    |\sigma^{(j)}(X;T_1,\ldots,T_j)|
    \leq
    C_j M(X)
    \prod_{i=1}^j g_X(T_i)^{1/2}
\end{equation}
for all $X,T_1,\ldots,T_j\in\mathbb R^{2n}$.

Here $\sigma^{(j)}(X;T_1,\ldots,T_j)$ denotes the $j$-th differential of $\sigma$ at $X$ evaluated in the directions $T_1,\ldots,T_j$. The space $S(M,g)$ becomes a Fr\'echet space when endowed with the seminorms given by the smallest constants $C_j$ satisfying \eqref{inwhk}.
\end{definition}

We now introduce the polynomial classes that determine the anharmonic oscillators studied in this paper. Let $m\geq1$ be an integer. We denote by $\mathcal P_{2m}$ the class of real-valued polynomials $P$ on $\mathbb R^n$ such that
\[
    \liminf_{|x|\to\infty}\frac{P(x)}{|x|^{2m}}>0 .
\]
Let $P\in\mathcal P_{2k},
    \,\,
    Q\in\mathcal P_{2\ell},$ for $
    k,\ell\geq1.$
We consider operators of the form
\begin{equation}
\label{qpop5}
    T=Q(D)+P(x).
\end{equation}
Since $P\in\mathcal P_{2k}$, there exists $p_0>0$ such that $P(x)+p_0>0$ for all $x\in\mathbb R^n$. Similarly, there exists $q_0>0$ such that $Q(\xi)+q_0>0$ for all $\xi\in\mathbb R^n$. Associated with $T$ we consider the metric
\begin{equation}
\label{anhm01}
    g=g^{(P,Q)}
    =
    \frac{dx^2}
    {(p_0+q_0+P(x)+Q(\xi))^{1/k}}
    +
    \frac{d\xi^2}
    {(p_0+q_0+P(x)+Q(\xi))^{1/\ell}} .
\end{equation}
It was proved in \cite{anh:cdr, anh:cdr2} that \eqref{anhm01} is a H\"ormander metric. Moreover, the symbol $P(x)+Q(\xi)$ belongs to $S(M,g)$ with
\[
    M(x,\xi)=p_0+q_0+P(x)+Q(\xi).
\]
If $k=\ell$, then the metric $g^{(P,Q)}$ is symplectic.

The preceding H\"ormander symbol class admits the following intrinsic description in terms of anisotropic derivatives.

\begin{definition}
\label{DEF:sigmas}
Let $m\in\mathbb R$, $P\in\mathcal P_{2k}$ and $Q\in\mathcal P_{2\ell}$, with $k,\ell\geq1$. We say that $a\in C^\infty(\mathbb R^n\times\mathbb R^n)$
belongs to $\Sigma_{P,Q}^m$ if, for some sufficiently large $q>0$ and for every pair of multi-indices $\alpha,\beta$, there exists a constant $C_{\alpha,\beta}>0$ such that
\begin{equation}
\label{sigmacl}
    |\partial_x^\beta\partial_\xi^\alpha a(x,\xi)|
    \leq
    C_{\alpha,\beta}
    \big(q+P(x)+Q(\xi)\big)^{
    \frac m2-\frac{|\beta|}{2k}-\frac{|\alpha|}{2\ell}}
\end{equation}
for all $(x,\xi)\in\mathbb R^n\times\mathbb R^n$.
\end{definition}

This class agrees with the corresponding Weyl--H\"ormander class associated with the metric $g^{(P,Q)}$. Indeed, we have 

\begin{proposition}
\label{sht57}
Let $m\in\mathbb R$ and let $g=g^{(P,Q)}$ be the metric defined in \eqref{anhm01}. Then
\[
    \Sigma_{P,Q}^m
    =
    S\left(
    \lambda_g^{m\frac{k\ell}{k+\ell}},
    g
    \right).
\]
\end{proposition}

\section{Mapping properties of the anharmonic fractional heat semigroup on modulation spaces}
\label{Modul}

In this section we study boundedness properties on modulation spaces for a class of pseudo-differential operators associated with anharmonic oscillators. In particular, we consider operators related to negative powers of anharmonic Hamiltonians and to fractional heat semigroups generated by positive powers of such Hamiltonians. We first introduce the modulation weights adapted to the operator $\mathcal H_{P,Q}$ and then recall the relevant symbolic and spectral properties.

Let $P\in\mathcal P_{2k}$ and $Q\in\mathcal P_{2\ell}$, where $k,\ell\geq1$. For $q>0$ sufficiently large, define
\[
    \widetilde v(x,\xi)
    =
    \big(q+P(x)+Q(\xi)\big)^{1/2}.
\]
This weight is equivalent to the anisotropic polynomial weight $1+|x|^k+|\xi|^\ell.$
More precisely, there exists a constant $C_1>1$ such that
$$C_1^{-1}\big(1+|x|^k+|\xi|^\ell\big)  \leq
    \big(q+P(x)+Q(\xi)\big)^{1/2} \leq
    C_1\big(1+|x|^k+|\xi|^\ell\big).$$
Furthermore, the model weight $1+|x|^k+|\xi|^\ell$ is submultiplicative up to a multiplicative constant. Indeed, for all $x,y,\xi,\eta\in\mathbb R^n$ there exists $C>1$ such that
\begin{align*}
\big(1+|x|^k+|\xi|^\ell\big)
 \big(1+|y|^k+|\eta|^\ell\big)  \geq
1+|x|^k+|y|^k+|\xi|^\ell+|\eta|^\ell  \geq
C^{-1}
\big(1+|x+y|^k+|\xi+\eta|^\ell\big).
\end{align*}
Thus, after replacing $\widetilde v$ by an equivalent weight $v=C\widetilde v$, we may assume that $v$ is submultiplicative. Throughout this section, we work with modulation spaces weighted by powers of $v$, namely, weights of the form $w=v^m,\,\, m\in\mathbb R.$
When $m=0$, we simply write $\mathcal M^{p,q}$ instead of $\mathcal M_{v^0}^{p,q}$.

For the metric $g=g^{(P,Q)}$, the uncertainty parameter is $ \lambda_g(x,\xi)=\big(q+P(x)+Q(\xi)\big)^{\frac{k+\ell}{2k\ell}}.$
Consequently, the Planck function and the modulation weight are related by
\begin{equation}
    h_g^{-1}
    =
    \lambda_g
    \simeq
    v^{\frac{k+\ell}{k\ell}} .
\end{equation}
The exponent $(k+\ell)/(k\ell)$ will also appear in the spectral asymptotics of the corresponding anharmonic oscillators.

The symbol class $\Sigma_{P,Q}^m$ is a Fr\'echet space when equipped with the seminorms
\[
    \Vert a\Vert_{\alpha,\beta}
    :=
    \inf\{C_{\alpha,\beta}:\ \eqref{sigmacl}\ \text{holds}\}.
\]
It is also useful to relate these classes to the global symbolic classes of Nicola and Rodino \cite{NicolaRodino}. Let $   k_0=\max\{k,\ell\}.$
Then
\begin{equation}
\label{anhmet012hh}
    g^{(P,Q)}
    \leq
    \frac{dx^2}{(q+P(x)+Q(\xi))^{1/k_0}}
    +
    \frac{d\xi^2}{(q+P(x)+Q(\xi))^{1/k_0}} .
\end{equation}
Taking
\[
    \Phi(x,\xi)=\Psi(x,\xi)
    =
    \big(q+P(x)+Q(\xi)\big)^{1/k_0},
    \qquad
    M(x,\xi)=q+P(x)+Q(\xi),
\]
we obtain sublinear and temperate weights $\Phi$ and $\Psi$, while $M$ is temperate. Hence, in the notation of \cite{NicolaRodino},
\begin{equation}
\label{ineqclassa}
    S\left(
    \big(q+P(x)+Q(\xi)\big)^{s/2},
    g
    \right)
    \subset
    S(M^{s/2};\Phi,\Psi).
\end{equation}

We now introduce the anharmonic oscillator $\mathcal H_{P,Q} =
    r_0+Q(D)+P(x),$
where $r_0\geq0$ is chosen so that $\mathcal H_{P,Q}$ is strictly positive. The operator is initially defined on a dense domain in $L^2(\mathbb R^n)$ and its closure has purely discrete spectrum. More precisely, by \cite[Theorem 4.2.9]{NicolaRodino}, the spectrum consists of a sequence of positive eigenvalues tending to infinity. Each eigenvalue has finite multiplicity, the corresponding eigenfunctions belong to $\mathcal S(\mathbb R^n)$, and these eigenfunctions form an orthonormal basis of $L^2(\mathbb R^n)$. In particular, $0$ belongs to the resolvent set of the closure of $\mathcal H_{P,Q}$.

The above anharmonic oscillators are $g$-elliptic in the sense of \cite{bufa:hy}. Hence one can define real powers $\mathcal H_{P,Q}^s$ for all $s\in\mathbb R$. Moreover, the strong uncertainty principle holds in this setting. For $s\in\mathbb R$, we define the generalized anharmonic-Sobolev space
\[
    H_{P,Q}^s(\mathbb R^n)
    =
    \left\{
    u\in\mathcal S'(\mathbb R^n):
    \Vert \mathcal H_{P,Q}^{s/2}u\Vert_{L^2(\mathbb R^n)}<\infty
    \right\}.
\]
This space is a Hilbert space with an inner product
\[
    \langle u,v\rangle_{H_{P,Q}^s}
    =
    \left\langle
    \mathcal H_{P,Q}^{s/2}u,
    \mathcal H_{P,Q}^{s/2}v
    \right\rangle_{L^2(\mathbb R^n)} .
\]
Its dual can be identified with $H_{P,Q}^{-s}(\mathbb R^n)$. In the special case
\[
    Q(\xi)=|\xi|^2,\qquad P(x)=|x|^2,
\]
these spaces coincide with the classical Shubin--Sobolev, or Hermite--Sobolev, spaces.

The generalized anharmonic-Sobolev spaces can also be described in terms of modulation spaces. By adapting the proof of \cite[Lemma 4.4.19]{ElenaRodino} and using the embedding and duality properties of $H_{P,Q}^m(\mathbb R^n)$, one obtains the following characterization.

\begin{lemma}
\label{char:ansobo}
For every $m\in\mathbb R$,
\[
    \mathcal M_{v^m}^{2,2}(\mathbb R^n)
    =
    H_{P,Q}^m(\mathbb R^n)
\]
with equivalent norms.
\end{lemma}

Combining Lemma \ref{char:ansobo} with the standard inclusion relations for modulation spaces and H\"older's inequality, we obtain that, for every $0<p,q\leq\infty$ and for $m$ sufficiently large,
\begin{equation}
\label{embeddings}
    H_{P,Q}^m(\mathbb R^n)
    \hookrightarrow
    \mathcal M^{p,q}(\mathbb R^n)
    \hookrightarrow
    \mathcal M^{\infty,\infty}(\mathbb R^n)
    \hookrightarrow
    H_{P,Q}^{-m}(\mathbb R^n).
\end{equation}

Let $\{\lambda_j\}_{j=0}^\infty$ denote the eigenvalues of $\mathcal H_{P,Q}$, ordered increasingly and repeated according to multiplicity, with $\lambda_0>0$. The spectral asymptotics obtained in \cite{anh:cdr2} yield
\begin{equation}
\label{specA}
    \lambda_j
    \sim
    C_{k,\ell}
    j^{\frac{2k\ell}{n(k+\ell)}},
    \qquad j\to\infty .
\end{equation}
Let $\{\Phi_j\}$ be an orthonormal basis of eigenfunctions. If $H_j$ denotes the eigenspace associated with $\lambda_j$ and $P_j$ is the corresponding orthogonal projection, then
\[
    P_jf
    =
    \sum_{i=1}^{d_j}
    \langle f,\Phi_{j_i}\rangle \Phi_{j_i},
\]
where $d_j=\dim H_j$. Hence,
\[
    \mathcal H_{P,Q}f
    =
    \sum_{j=0}^\infty
    \lambda_j P_jf .
\]
By the spectral theorem, the fractional powers are given by
\[
    \mathcal H_{P,Q}^s f
    =
    \sum_{j=0}^\infty
    \lambda_j^s P_jf,
    \qquad s\in\mathbb R.
\]
Equivalently,
\[
    \Vert f\Vert_{H_{P,Q}^s}^2
    =
    \sum_{j=0}^\infty
    \lambda_j^s
    \Vert P_jf\Vert_{L^2(\mathbb R^n)}^2 .
\]

We shall also use the following elementary identity, which is useful in the analysis of modulation-space boundedness.

\begin{lemma}
\label{lempq3s}
Let $P\in\mathcal P_{2k}$ and $Q\in\mathcal P_{2\ell}$, where $k,\ell\geq1$. If $Y=(y,\eta),\,\, W=(w,\tau)\in\mathbb R^{2n},$
and $N\in\mathbb N$, then
\begin{equation}
\label{modexp1axx}
    e^{-iW\cdot Y}
    =
    \big(q+P(w)+Q(\tau)\big)^{-N}
    \left(
    \frac q2+P_y(D)+\frac q2+Q_\eta(D)
    \right)^N
    e^{-iW\cdot Y}.
\end{equation}
Here $P_y(D)$ denotes the operator $P(D)$ acting in the $y$-variable, while $Q_\eta(D)$ acts in the $\eta$-variable.
\end{lemma}

\begin{proof}
The identity follows immediately from the action of constant-coefficient differential operators on exponential functions. Indeed,
\[
    \left(
    \frac q2+P_y(D)+\frac q2+Q_\eta(D)
    \right)^N
    e^{-iW\cdot Y}
    =
    \big(q+P(w)+Q(\tau)\big)^N
    e^{-iW\cdot Y},
\]
which is equivalent to \eqref{modexp1axx}.
\end{proof}

We can now state the boundedness result on modulation spaces for $t$-quantizations of symbols in the classes introduced above.

\begin{theorem}\label{modb12}
Let $m\in\mathbb{R}$ and let $a\in\Sigma_{P,Q}^{-m}$, where $P,Q$ are as above. Let $0<p,q\leq\infty$ and $t\in\mathbb{R}$. Then $a_t(x,D):\M^{p,q}(\mathbb{R}^n)\longrightarrow \M_{v^m}^{p,q}(\mathbb{R}^n)$
extends to a bounded operator, where $v(x,\xi):=(q_0+P(x)+Q(\xi))^{1/2}$
for $q_0>0$ sufficiently large. Moreover, the operator norm depends only on finitely many seminorms of $a$ in $\Sigma_{P,Q}^{-m}$.
\end{theorem}

\begin{proof}
Let $v(x,\xi):=(q_0+P(x)+Q(\xi))^{1/2},$
where $q_0>0$ is chosen sufficiently large. We consider the modulation weights $\omega_1(x,\xi)=1,
\,\,
\omega_2(x,\xi)=v(x,\xi)^m=(q_0+P(x)+Q(\xi))^{m/2}.$
By Theorem 3.1 of \cite{toft1:mod}, the boundedness $a_t(x,D):\M^{p,q}\rightarrow \M_{v^m}^{p,q}$
will follow once we prove that the symbol $a$ belongs to a suitable modulation space of symbols, $a\in \M_{\omega_0}^{\infty,r}(\mathbb{R}^{2n}),$
for some $r\leq \min\{1,p,q\}$, with a weight $\omega_0$ compatible with $\omega_1$ and $\omega_2$. In the present case, it is enough to prove the estimate
\[
|V_g a(Z,W)|
\leq
C_N v(Z)^{-m}
(1+|w|^k+|\tau|^\ell)^{-2N},
\]
for all $N\in\mathbb{N}$ sufficiently large, where $Z=(z,\gamma),\,\, W=(w,\tau)\in\mathbb{R}^{2n}.$

We now prove this estimate. For $Y=(y,\eta)$ and $W=(w,\tau)$, Lemma \ref{lempq3s} gives
\begin{equation}\label{modexp1a}
e^{-iW\cdot Y}
=
(q_0+P(w)+Q(\tau))^{-N}
\left(
\frac{q_0}{2}+P_y(D)+\frac{q_0}{2}+Q_\eta(D)
\right)^N
e^{-iW\cdot Y}.
\end{equation}
Here $P_y(D)$ acts in the $y$-variable and $Q_\eta(D)$ acts in the $\eta$-variable.

We also recall that, since $P\in\mathcal P_{2k}$ and $Q\in\mathcal P_{2\ell}$, for $q_0>0$ sufficiently large,
\[
q_0+P(x)+Q(\xi)
\asymp
1+|x|^{2k}+|\xi|^{2\ell}.
\]
Consequently,
\begin{equation}\label{in453}
v(x,\xi)
=
(q_0+P(x)+Q(\xi))^{1/2}
\asymp
1+|x|^k+|\xi|^\ell.
\end{equation}

Let $g\in\mathcal S(\mathbb{R}^{2n})$ be a fixed nonzero window. The short-time Fourier transform of the symbol $a$ is
\[
V_g a(Z,W)
=
\int_{\mathbb{R}^{2n}}
a(Y)\overline{g(Y-Z)}e^{-iW\cdot Y}\,dY.
\]
Using \eqref{modexp1a}, followed by integration by parts, we transfer the differential operator in \eqref{modexp1a} from the exponential to the factor $a(Y)\overline{g(Y-Z)}.$
Since $a\in\Sigma_{P,Q}^{-m}$ and $g\in\mathcal S(\mathbb{R}^{2n})$, the derivatives produced in this integration by parts are controlled by the seminorms of $a$ in $\Sigma_{P,Q}^{-m}$ and by the rapid decay of $g$. Hence, for every $N\in\mathbb{N}$, there exists $C_N>0$ such that
\begin{equation}\label{tt67}
|V_g a(Z,W)|
\leq
C_N
v(Z)^{-m}
(q_0+P(w)+Q(\tau))^{-N}.
\end{equation}
Using again \eqref{in453}, we get $(q_0+P(w)+Q(\tau))^{-N}
\lesssim
(1+|w|^k+|\tau|^\ell)^{-2N}.$
Therefore,
\[
|V_g a(Z,W)|
\leq
C_N
v(Z)^{-m}
(1+|w|^k+|\tau|^\ell)^{-2N}.
\]

Choosing $N$ large enough, for instance $N>\frac{n}{r\min\{k,\ell\}},$
ensures that the function $W\mapsto (1+|w|^k+|\tau|^\ell)^{-2N}$
belongs to $L^r(\mathbb{R}^{2n})$. Hence $a\in \M_{\omega_0}^{\infty,r}(\mathbb{R}^{2n})$
with the required symbol weight. The desired boundedness of $a_t(x,D)$ now follows from Theorem 3.1 of \cite{toft1:mod}. This completes the proof.
\end{proof}

The next result, concerning the symbolic structure of the fractional anharmonic oscillator, will be essential in the proof of the main result of this section.

\begin{theorem}\label{regansymb}
Let $s>0$. For $r_0\geq0$ sufficiently large, the fractional anharmonic oscillator $\mathcal{H}_{P,Q}^s:=(Q(D)+P(x)+r_0)^s$
is a pseudo-differential operator with real Weyl symbol $a_s(x,\xi):=\mathcal{H}_{P,Q}^s(x,\xi)\in \Sigma_{P,Q}^{2s}.$
More precisely, for $|\xi|^\ell+|x|^k\geq 1$, the symbol has the form
\begin{equation}\label{inet56}
\mathcal{H}_{P,Q}^s(x,\xi)
=
(Q(\xi)+P(x)+r_0)^s+r(x,\xi),
\end{equation}
where $r\in \Sigma_{P,Q}^{2s-\frac{k+\ell}{k\ell}}.$
\end{theorem}

\begin{proof}
Let $p(x,\xi):=Q(\xi)+P(x)+r_0.$
For $r_0\geq0$ sufficiently large, the operator $\mathcal H_{P,Q}=Q(D)+P(x)+r_0$
is strictly positive. Moreover, its Weyl symbol is precisely $p(x,\xi)=Q(\xi)+P(x)+r_0,$
because $Q(D)$ depends only on $\xi$ and $P(x)$ depends only on $x$.

By the definition of the class $\Sigma_{P,Q}^{m}$, we have $p\in \Sigma_{P,Q}^{2}.$
Furthermore, $p$ is elliptic in the class associated to the metric $g^{(P,Q)}$. Indeed, for $q_0>0$ sufficiently large,
\[
q_0+P(x)+Q(\xi)
\asymp
1+|x|^{2k}+|\xi|^{2\ell},
\]
and hence $p(x,\xi)$ controls the natural order function of the calculus.

Recall that the uncertainty parameter of $g=g^{(P,Q)}$ is $\lambda_g(x,\xi)
=
(q_0+P(x)+Q(\xi))^{\frac{k+\ell}{2k\ell}},$
and therefore the Planck function is
\[
h_g(x,\xi)
=
\lambda_g(x,\xi)^{-1}
=
(q_0+P(x)+Q(\xi))^{-\frac{k+\ell}{2k\ell}}.
\]
Thus, one step in the symbolic expansion lowers the order by $\frac{k+\ell}{k\ell}.$ Since $\mathcal H_{P,Q}$ is a positive elliptic operator in the Weyl--H\"ormander calculus, the complex powers theorem for globally elliptic pseudo-differential operators, see \cite[Theorem 4.3.6]{NicolaRodino}, applies. It follows that $\mathcal H_{P,Q}^s$
is a pseudo-differential operator with Weyl symbol $a_s\in \Sigma_{P,Q}^{2s}.$
Moreover, the principal term of the symbol is $p(x,\xi)^s=(Q(\xi)+P(x)+r_0)^s,$
and the first remainder is lower by one power of the Planck function. Hence
\[
a_s(x,\xi)
-
(Q(\xi)+P(x)+r_0)^s
\in
\Sigma_{P,Q}^{2s-\frac{k+\ell}{k\ell}},
\]
which proves \eqref{inet56}.
\end{proof}

As a consequence of Theorem \ref{regansymb}, for every $s>0$ the operator $\mathcal H_{P,Q}^s$ is a pseudo-differential operator with positive elliptic Weyl symbol in the anharmonic symbol class $\Sigma_{P,Q}^{2s}$. Moreover, the spectrum of $\mathcal H_{P,Q}$ consists of a discrete sequence of eigenvalues diverging to infinity. If $\{\lambda_j\}_{j\geq0}$ denotes the eigenvalues, arranged in increasing order and counted with multiplicities, and if $P_j$ denotes the corresponding spectral projections, then for $t\geq0$ we define the fractional anharmonic heat semigroup by
\[
e^{-t\APQ^s}f
=
\sum_{j=0}^{\infty}
e^{-t\lambda_j^s}P_jf.
\]
Furthermore, by the heat kernel construction in the Weyl--H\"ormander calculus, see \cite[Theorem 4.5.1]{NicolaRodino}, the operator $e^{-t\APQ^s}$ is itself a pseudo-differential operator whose Weyl symbol satisfies uniform estimates that will be used in the proof of the next theorem.

We conclude this section with the main mapping estimate (Theorem \ref{mainthmest}) for the anharmonic fractional heat semigroup on modulation spaces. The proof is similar to that of \cite{Cardona}, but we discuss it here only for the short-time estimates in our general setting.

\begin{proof}[Proof of Theorem \ref{mainthmest}]
We divide the proof into two cases.

\medskip
\noindent\textbf{Case $0<t\leq 1$.}
We first prove the estimate in the small-time regime. By the embeddings of modulation spaces with respect to the exponents, it is enough to prove the result after replacing $p_2$ by $\min\{p_1,p_2\}$ and $q_2$ by $\min\{q_1,q_2\}$. Hence, throughout this part of the proof, we may assume without loss of generality that $p_2\leq p_1,\,\,
q_2\leq q_1.$
Consequently,
\[
\frac1{p_2}=\frac1{p_1}+\frac1{\widetilde p},
\qquad
\frac1{q_2}=\frac1{q_1}+\frac1{\widetilde q},
\]
with the usual interpretation when $\widetilde p=\infty$ or $\widetilde q=\infty$.

By Theorem \ref{regansymb}, the fractional anharmonic oscillator $\APQ^s$ is a pseudo-differential operator with real Weyl symbol $a_s(x,\xi)\in \Sigma_{P,Q}^{2s}.$
Moreover, $a_s$ is elliptic in the Weyl--H\"ormander calculus associated with the metric $g^{(P,Q)}$. Therefore, applying the heat kernel construction for globally elliptic pseudo-differential operators, see \cite[Theorem 4.5.1]{NicolaRodino}, the semigroup $e^{-t\APQ^s}$ is a pseudo-differential operator with Weyl symbol $u_t(x,\xi)$ such that, for every $N\geq 0$, $t^Nu_t$
belongs to a bounded subset of $\Sigma_{P,Q}^{-2sN}$, uniformly for $0<t\leq 1$.

Let $v(x,\xi):=(q+P(x)+Q(\xi))^{1/2},$
where $q>0$ is chosen sufficiently large. Since $t^Nu_t\in \Sigma_{P,Q}^{-2sN},$
Theorem \ref{modb12} gives the boundedness $(t^Nu_t)^w(x,D):
\M^{p_1,q_1}(\mathbb{R}^n)
\longrightarrow
\M_{v^{2sN}}^{p_1,q_1}(\mathbb{R}^n),$
with operator norm independent of $t\in(0,1]$. Hence $\Vert t^N e^{-t\APQ^s}f\Vert_{\M_{v^{2sN}}^{p_1,q_1}}
\leq
C\Vert f\Vert_{\M^{p_1,q_1}}.$
Together with the boundedness of $e^{-t\APQ^s}$ on $\M^{p_1,q_1}$, this yields
\begin{equation}
\label{weighted-stft-bound}
\Vert
(1+t^Nv^{2sN})V_g(e^{-t\APQ^s}f)
\Vert_{L^{p_1,q_1}}
\leq
C
\Vert f\Vert_{\M^{p_1,q_1}},
\end{equation}
uniformly for $0<t\leq 1$.

We now prove the embedding estimate
\begin{equation}
\label{aux456}
\Vert F\Vert_{L^{p_2,q_2}}
\leq
C t^{-\sigma}
\Vert
(1+t^Nv^{2sN})F
\Vert_{L^{p_1,q_1}},
\end{equation}
for every measurable function $F$ on $\mathbb{R}^{2n}$. By H\"older's inequality in mixed-norm spaces,
\begin{align*}
\Vert F\Vert_{L^{p_2,q_2}}
&=
\left\Vert
(1+t^Nv^{2sN})^{-1}
(1+t^Nv^{2sN})F
\right\Vert_{L^{p_2,q_2}}
\\
&\leq
\left\Vert
(1+t^Nv^{2sN})^{-1}
\right\Vert_{L^{\widetilde p,\widetilde q}}
\left\Vert
(1+t^Nv^{2sN})F
\right\Vert_{L^{p_1,q_1}}.
\end{align*}

Thus, it remains to estimate \[ \widetilde C(t) := \left\Vert (1+t^N v^{2sN})^{-1} \right\Vert_{L^{\widetilde p,\widetilde q}}. \] At this point we use only the coercivity assumptions \( P\in\mathcal P_{2k} \) and \( Q\in\mathcal P_{2\ell}. \) Choosing \(q>0\) sufficiently large, there exists a constant \(c>0\) such that \[ q+P(x)+Q(\xi) \geq c\bigl(1+|x|^{2k}+|\xi|^{2\ell}\bigr), \qquad x,\xi\in\mathbb R^n. \] Hence \[ v(x,\xi) = (q+P(x)+Q(\xi))^{1/2} \gtrsim 1+|x|^k+|\xi|^\ell, \] and therefore \[ v(x,\xi)^{2sN} \gtrsim (1+|x|^k+|\xi|^\ell)^{2sN}. \] Consequently, \[ (1+t^N v(x,\xi)^{2sN})^{-1} \lesssim \Bigl( 1+t^N(1+|x|^k+|\xi|^\ell)^{2sN} \Bigr)^{-1}, \] so that \[ \widetilde C(t) \lesssim \left\Vert \Bigl( 1+t^N(1+|x|^k+|\xi|^\ell)^{2sN} \Bigr)^{-1} \right\Vert_{L^{\widetilde p,\widetilde q}}. \]

We now use the anisotropic scaling. Since $t^N(1+|x|^k+|\xi|^\ell)^{2sN}
=
\left(
t^{1/(2s)}(1+|x|^k+|\xi|^\ell)
\right)^{2sN},$
we make the change of variables $X=t^{\frac1{2sk}}x,\,
\Xi=t^{\frac1{2s\ell}}\xi.$
Then $dx=t^{-\frac n{2sk}}\,dX,
\,\,
d\xi=t^{-\frac n{2s\ell}}\,d\Xi,$
and $t^{1/(2s)}|x|^k=|X|^k,
\,\,
t^{1/(2s)}|\xi|^\ell=|\Xi|^\ell.$
Moreover,
\[
t^{1/(2s)}
(1+|x|^k+|\xi|^\ell)
=
t^{1/(2s)}+|X|^k+|\Xi|^\ell
\geq
|X|^k+|\Xi|^\ell.
\]

Assume first that $\widetilde p,\widetilde q<\infty$. Then
\begin{align*}
&\left\Vert
\left(
1+t^N(1+|x|^k+|\xi|^\ell)^{2sN}
\right)^{-1}
\right\Vert_{L^{\widetilde p,\widetilde q}}^{\widetilde q}
\\
&\quad\leq
C
t^{-\frac{n\widetilde q}{2sk\widetilde p}}
t^{-\frac n{2s\ell}}
\int_{\mathbb{R}^n}
\left(
\int_{\mathbb{R}^n}
\left(
1+(|X|^k+|\Xi|^\ell)^{2sN}
\right)^{-\widetilde p}
dX
\right)^{\widetilde q/\widetilde p}
d\Xi.
\end{align*}
Choosing $N$ sufficiently large, the last integral is finite. Hence
\[
\left\Vert
\left(
1+t^N(1+|x|^k+|\xi|^\ell)^{2sN}
\right)^{-1}
\right\Vert_{L^{\widetilde p,\widetilde q}}
\leq
C
t^{-\frac n{2sk\widetilde p}}
t^{-\frac n{2s\ell\widetilde q}}.
\]
That is,
\[
\widetilde C(t) \lesssim \left\Vert
\left(
1+t^N(1+|x|^k+|\xi|^\ell)^{2sN}
\right)^{-1}
\right\Vert_{L^{\widetilde p,\widetilde q}}
\leq
C
t^{-\frac n{2s}
\left(
\frac1{k\widetilde p}
+
\frac1{\ell\widetilde q}
\right)}
=
Ct^{-\sigma}.
\]
The endpoint cases $\widetilde p=\infty$ or $\widetilde q=\infty$ are handled in the same way, with the corresponding essential suprema replacing the relevant integrals. Thus
\[
\widetilde C(t)\leq Ct^{-\sigma},
\qquad
0<t\leq 1.
\]

Consequently, \eqref{aux456} holds. Applying \eqref{aux456} with $F=V_g(e^{-t\APQ^s}f)$
and using \eqref{weighted-stft-bound}, we obtain
\[
\Vert e^{-t\APQ^s}f\Vert_{\M^{p_2,q_2}}
\leq
C t^{-\sigma}
\Vert f\Vert_{\M^{p_1,q_1}},
\qquad
0<t\leq 1.
\]
This proves the desired estimate in the small-time regime.

\medskip

\noindent\textbf{Case $t\geq1$.} The proof is exactly similar to \cite{Cardona} so we will not repeat it here.

This completes the proof.
\end{proof}

\section{Mapping properties of anharmonic fractional heat semigroups on Lebesgue spaces}
\label{Lpprop}

In this section, we present the proof of Theorem \ref{Lpmapping}, i.e.,   the $L^p$-$L^q$ time-decay of the fractional heat semigroups $e^{-t\APQ^s}.$

\begin{proof}[Proof of Theorem \ref{Lpmapping}]
We first prove the estimate for large time, namely $t>1$. We use the standard embeddings
\begin{equation}
\label{inh9dj}
    L^p(\mathbb R^n)\hookrightarrow \mathcal M^{p,\infty}(\mathbb R^n),
    \qquad
    \mathcal M^{q,1}(\mathbb R^n)\hookrightarrow L^q(\mathbb R^n),
    \qquad 1\leq p,q\leq\infty .
\end{equation}
Applying Theorem \ref{mainthmest} with $p_1=p,\,\, p_2=q,\,\, q_1=\infty,\,\, q_2=1,$
we obtain, for $t>1$,
\[
\begin{aligned}
    \Vert e^{-t\APQ^s}f\Vert_{L^q}
    \lesssim
    \Vert e^{-t\APQ^s}f\Vert_{\mathcal M^{q,1}}       \leq
    C e^{-t\lambda_0^s}
    \Vert f\Vert_{\mathcal M^{p,\infty}}                
    \lesssim
    C e^{-t\lambda_0^s}
    \Vert f\Vert_{L^p}.
\end{aligned}
\]
This proves \eqref{lpest}.

We now consider the case $0<t\leq1$. By Theorem \ref{regansymb}, the operator $\APQ^s$ is a pseudo-differential operator whose Weyl symbol belongs to $\Sigma_{P,Q}^{2s}$. More precisely, its principal part has the form $a_s(x,\xi)
    =
    \big(Q(\xi)+P(x)+r_0\big)^s
    +
    R(x,\xi),$
where $R\in
    \Sigma_{P,Q}^{2s-\frac{k+\ell}{k\ell}}.$
Passing from Weyl to Kohn--Nirenberg quantization only changes the symbol by lower-order terms. Hence, we may write the corresponding Kohn--Nirenberg symbol as
\[
    a_s(x,\xi)=a(x,\xi)+R'(x,\xi),
\]
where $R'$ is of lower order and where $a\in\Sigma_{P,Q}^{2s}$ satisfies the elliptic lower bound
\begin{equation}
\label{lw29a}
    a(x,\xi)
    \geq
    c\big(1+Q(\xi)+P(x)\big)^s
    \geq
    c\big(1+|x|^{2k}+|\xi|^{2\ell}\big)^s .
\end{equation}
In particular,
\[
    a(x,\xi)
    \geq
    c_1 |x|^{2k s}+c_2|\xi|^{2\ell s}-C .
\]

By the heat parametrix expansion in the global calculus, see \cite[Theorem 4.5.1]{NicolaRodino}, the Kohn--Nirenberg symbol of $e^{-t\APQ^s}$ can be written, for arbitrary $J\geq1$, as
\[
    u_t(x,\xi)
    =
    b_t(x,\xi)+R_{t,J}(x,\xi),
\]
where
\[
    b_t(x,\xi)
    =
    e^{-t a(x,\xi)}
    \left(
    1+
    \sum_{j=1}^{J-1}
    \sum_{\ell=1}^{2j}
    t^\ell u_{\ell,j}(x,\xi)
    \right),
\]
with
\[
    u_{\ell,j}\in
    \Sigma_{P,Q}^{2s-\frac{k+\ell}{k\ell}j},
\]
and the remainder satisfies, uniformly for $0<t\leq1$, $R_{t,J}\in \Sigma_{P,Q}^{-2J}.$
Choosing $J$ sufficiently large, the operator $R_{t,J}(x,D)$ is smoothing enough to satisfy
\begin{equation}
\label{remLpLq}
    \Vert R_{t,J}(x,D)f\Vert_{L^q}
    \leq
    C\Vert f\Vert_{L^p},
    \qquad 1\leq p,q\leq\infty .
\end{equation}
Since $t^{-\sigma_s}\geq1$ for $0<t\leq1$, the contribution of the remainder is always dominated by the right-hand side of \eqref{lpest1}. Thus it remains to estimate the operator with symbol $b_t$.

We first prove the $L^1\to L^\infty$ estimate. From \eqref{lw29a} and from the symbolic estimates for the coefficients $u_{\ell,j}$, we have
\[
    |b_t(x,\xi)|
    \leq
    C e^{-c t|\xi|^{2\ell s}},
    \qquad 0<t\leq1.
\]
Let $K_t(x,y)$ denote the integral kernel of $b_t(x,D)$. Then
\[
    K_t(x,y)
    =
    (2\pi)^{-n}
    \int_{\mathbb R^n}
    e^{i(x-y)\cdot\xi}
    b_t(x,\xi)\,d\xi .
\]
Therefore,
\[
\begin{aligned}
    |K_t(x,y)|
    &\leq
    C\int_{\mathbb R^n}
    e^{-c t|\xi|^{2\ell s}}\,d\xi          \\
    &\leq
    C t^{-\frac{n}{2\ell s}}.
\end{aligned}
\]
It follows that
\begin{equation}
\label{L1Linf}
    \Vert b_t(x,D)f\Vert_{L^\infty}
    \leq
    C t^{-\frac{n}{2\ell s}}
    \Vert f\Vert_{L^1}.
\end{equation}
Together with \eqref{remLpLq}, this proves case {\it (ii)}.

Next we prove the estimates in the region $q\leq p$. We first claim that, for every pair of multi-indices $\alpha,\beta$,
\begin{equation}
\label{intxo76}
    \left|
    \partial_x^\alpha\partial_\xi^\beta
    \left[
    e^{\frac{t}{4}\langle x\rangle^{2k s}}
    b_t(x,\xi)
    \right]
    \right|
    \leq
    C_{\alpha,\beta}
    (1+|\xi|)^{-|\beta|},
    \qquad 0<t\leq1.
\end{equation}
Indeed, derivatives of the weight satisfy
\[
    \left|
    \partial_x^\alpha
    e^{\frac{t}{4}\langle x\rangle^{2k s}}
    \right|
    \leq
    C_\alpha
    e^{\frac{t}{2}\langle x\rangle^{2k s}}
    \langle x\rangle^{-|\alpha|}.
\]
On the other hand, since $a\in\Sigma_{P,Q}^{2s}$,
\[
    \left|
    \partial_x^\alpha\partial_\xi^\beta a(x,\xi)
    \right|
    \leq
    C_{\alpha,\beta}
    a(x,\xi)
    \big(q+P(x)+Q(\xi)\big)^{
    -\frac{|\alpha|}{2k}
    -\frac{|\beta|}{2\ell}} .
\]
Consequently,
\[
    \left|
    \partial_x^\alpha\partial_\xi^\beta
    e^{-t a(x,\xi)}
    \right|
    \leq
    C_{\alpha,\beta}
    e^{-\frac t2 a(x,\xi)}
    \big(q+P(x)+Q(\xi)\big)^{
    -\frac{|\alpha|}{2k}
    -\frac{|\beta|}{2\ell}} .
\]
Combining these estimates with the symbolic bounds for the coefficients
$u_{\ell,j}$ and using Leibniz' formula gives \eqref{intxo76}. Hence $e^{\frac{t}{4}\langle x\rangle^{2k s}}b_t(x,D)$
is a pseudo-differential operator with the symbol bounded uniformly in
$S_{1,0}^0$. Therefore, by the $L^p$ boundedness of operators with symbols in
$S_{1,0}^0$, we obtain, for $1<p<\infty$,
\begin{equation}
\label{ppCt}
    \Vert
    e^{\frac{t}{4}\langle x\rangle^{2k s}}
    b_t(x,D)f
    \Vert_{L^p}
    \leq
    C\Vert f\Vert_{L^p}.
\end{equation}

Let now $1\leq q\leq p<\infty$ with $p>1$. By H\"older's inequality,
\begin{equation}
\label{holder56}
    \Vert
    e^{-\frac{t}{4}\langle x\rangle^{2k s}}g
    \Vert_{L^q}
    \leq
    C t^{-\frac{n}{2k s}\left(\frac1q-\frac1p\right)}
    \Vert g\Vert_{L^p}.
\end{equation}
Indeed, if $q<p$, then
\[
    \frac1r=\frac1q-\frac1p,
\]
and
\[
    \Vert
    e^{-\frac{t}{4}\langle x\rangle^{2k s}}
    \Vert_{L^r}
    \leq
    C t^{-\frac{n}{2k s}\left(\frac1q-\frac1p\right)}.
\]
The case $q=p$ is trivial.

Applying \eqref{holder56} to
\[
    g=
    e^{\frac{t}{4}\langle x\rangle^{2k s}}
    b_t(x,D)f
\]
and using \eqref{ppCt}, we get
\[
    \Vert b_t(x,D)f\Vert_{L^q}
    \leq
    C
    t^{-\frac{n}{2k s}\left(\frac1q-\frac1p\right)}
    \Vert f\Vert_{L^p}.
\]
Together with the remainder estimate \eqref{remLpLq}, this proves
\[
    \Vert e^{-t\APQ^s}f\Vert_{L^q}
    \leq
    C
    t^{-\frac{n}{2k s}\left(\frac1q-\frac1p\right)}
    \Vert f\Vert_{L^p},
\]
for $1<p<\infty$ and $1\leq q\leq p$. In particular, taking $q=1$ gives case {\it (iv)}.

It remains to prove the estimates in the region $p\leq q$, with
$p,q\in(1,\infty)$. We already have the endpoint estimate
\[
    L^1\to L^\infty,
    \qquad
    \Vert e^{-t\APQ^s}f\Vert_{L^\infty}
    \leq
    C t^{-\frac{n}{2\ell s}}
    \Vert f\Vert_{L^1},
\]
and we also have, from the previous part with $p=q=r$,
\[
    \Vert e^{-t\APQ^s}f\Vert_{L^r}
    \leq
    C\Vert f\Vert_{L^r},
    \qquad 1<r<\infty.
\]
Interpolating between these two bounds by the Riesz--Thorin theorem gives
\[
    \Vert e^{-t\APQ^s}f\Vert_{L^q}
    \leq
    C
    t^{-\frac{n}{2\ell s}\left(\frac1p-\frac1q\right)}
    \Vert f\Vert_{L^p},
\]
for $1<p\leq q<\infty$. This completes the proof of case {\it (i)}.

Finally, we prove case {\it (iii)}. Let $f\in L^1(\mathbb R^n)\cap L^2(\mathbb R^n)$. Since the semigroup is self-adjoint on $L^2$,
\[
\begin{aligned}
    \Vert e^{-t\APQ^s}f\Vert_{L^2}^2
    =
    \left\langle
    e^{-t\APQ^s}f,
    e^{-t\APQ^s}f
    \right\rangle                         =
    \left\langle
    e^{-2t\APQ^s}f,
    f
    \right\rangle                         \leq
    \Vert e^{-2t\APQ^s}f\Vert_{L^\infty}
    \Vert f\Vert_{L^1}.
\end{aligned}
\]
Using the $L^1\to L^\infty$ estimate with time $2t$, we obtain $\Vert e^{-t\APQ^s}f\Vert_{L^2}^2
    \leq
    C t^{-\frac{n}{2\ell s}}
    \Vert f\Vert_{L^1}^2.$
Hence
\begin{equation}
\label{L1L2}
    \Vert e^{-t\APQ^s}f\Vert_{L^2}
    \leq
    C t^{-\frac{n}{4\ell s}}
    \Vert f\Vert_{L^1}.
\end{equation}
Interpolating \eqref{L1L2} with the $L^1\to L^\infty$ estimate
\eqref{L1Linf}, we obtain, for every $2\leq q<\infty$,
\[
    \Vert e^{-t\APQ^s}f\Vert_{L^q}
    \leq
    C
    t^{-\frac{n}{2\ell s}\left(1-\frac1q\right)}
    \Vert f\Vert_{L^1}.
\]
This proves case {\it (iii)} and completes the proof of the theorem.
\end{proof}

 \section{Local and global well-posedness and lower blow-up rate for nonlinear fractional heat equations associated with anharmonic oscillators}
\label{application}

In this section we study the nonlinear fractional heat equation associated with the anharmonic oscillator $\APQ$:
\begin{equation}
\label{eq6.1}
    \begin{cases}
        \partial_t u+\APQ^s u=|u|^{\beta-1}u,\\
        u(0,x)=u_0(x),
    \end{cases}
\end{equation}
for $(t,x)\in(0,\infty)\times\mathbb R^n$, where $\beta>1$ and $s>0$.

We first prove local well-posedness in the supercritical Lebesgue range.

\begin{theorem}
\label{them6.1}
Let $1<p<\infty$, $s>0$, $\beta>1$, and $u_0\in L^p(\mathbb R^n)$. Assume that
\[
    p>p_c^s:=\frac{n(\beta-1)}{2\ell s}.
\]
Then there exists $T>0$ such that \eqref{eq6.1} has a mild solution $ u\in C([0,T],L^p(\mathbb R^n)).$
Moreover, the solution extends to a maximal interval $[0,T_{\max})$ such that either $T_{\max}=\infty$, or $T_{\max}<\infty$ and $ \lim_{t\to T_{\max}}
    \Vert u(t)\Vert_{L^p(\mathbb R^n)}
    =
    \infty.$
\end{theorem}

\begin{proof}
Let $T>0$. We introduce the space
\[
    Z_T
    =
    \left\{
    u:(0,T)\to L^p(\mathbb R^n)\cap L^{p\beta}(\mathbb R^n):
    \Vert u\Vert_{Z_T}<\infty
    \right\},
\]
where
\[
    \Vert u\Vert_{Z_T}
    :=
    \max\left\{
    \sup_{0<t<T}\Vert u(t)\Vert_{L^p},
    \sup_{0<t<T}
    t^{\frac{n(\beta-1)}{2p\ell s\beta}}
    \Vert u(t)\Vert_{L^{p\beta}}
    \right\}.
\]
Let $M_1>0$ be such that $ \Vert u_0\Vert_{L^p} \leq M_1.$
By Theorem \ref{Lpmapping}, applied with $(p,q)=(p,p)$ and $(p,q)=(p,p\beta)$, we have $   \Vert e^{-t\APQ^s}u_0\Vert_{L^p}
    \leq
    C\Vert u_0\Vert_{L^p}$
and $ \Vert e^{-t\APQ^s}u_0\Vert_{L^{p\beta}}
    \leq
    C t^{-\frac{n}{2\ell s}\left(\frac1p-\frac1{p\beta}\right)}
    \Vert u_0\Vert_{L^p}.$
Since $ \frac{n}{2\ell s}\left(\frac1p-\frac1{p\beta}\right)
    =
    \frac{n(\beta-1)}{2p\ell s\beta},$
it follows that
\[
    \Vert e^{-t\APQ^s}u_0\Vert_{Z_T}
    \leq
    C M_1.
\]
Choose $M>0$ such that $C M_1\leq M.$
Define the closed ball
\[
    D_{M+1}
    :=
    \left\{
    u\in Z_T:
    \Vert u\Vert_{Z_T}\leq M+1
    \right\}.
\]
For $u\in D_{M+1}$, define
\[
    \Psi[u]
    :=
    e^{-t\APQ^s}u_0
    +
    \int_0^t
    e^{-(t-\tau)\APQ^s}
    \big(|u(\tau)|^{\beta-1}u(\tau)\big)\,d\tau.
\]
We shall prove that, for $T>0$ small enough, $\Psi$ is a contraction on $D_{M+1}$.

Let $q\in\{p,p\beta\}$. Since $\Vert |u(\tau)|^{\beta-1}u(\tau)\Vert_{L^p}
    =
    \Vert u(\tau)\Vert_{L^{p\beta}}^\beta,$
Theorem \ref{Lpmapping} gives
\[
\begin{aligned}
\left\Vert
\int_0^t
e^{-(t-\tau)\APQ^s}
\big(|u(\tau)|^{\beta-1}u(\tau)\big)\,d\tau
\right\Vert_{L^q}
\leq
C
\int_0^t
(t-\tau)^{-\frac{n}{2\ell s}\left(\frac1p-\frac1q\right)}
\Vert u(\tau)\Vert_{L^{p\beta}}^\beta
\,d\tau.
\end{aligned}
\]
Since $u\in D_{M+1}$, $\Vert u(\tau)\Vert_{L^{p\beta}}
    \leq
    (M+1)
    \tau^{-\frac{n(\beta-1)}{2p\ell s\beta}},$ it further implies that
\[
\begin{aligned}
\left\Vert
\int_0^t
e^{-(t-\tau)\APQ^s}
\big(|u(\tau)|^{\beta-1}u(\tau)\big)\,d\tau
\right\Vert_{L^q}
\leq
C(M+1)^\beta
\int_0^t
(t-\tau)^{-\frac{n}{2\ell s}\left(\frac1p-\frac1q\right)}
\tau^{-\frac{n(\beta-1)}{2p\ell s}}
\,d\tau.
\end{aligned}
\]
The integral is finite because $\frac{n(\beta-1)}{2p\ell s}<1,$
which is exactly the assumption $p>p_c^s$. Hence,

\begin{align} \label{lab5.2}
\left\Vert
\int_0^t
e^{-(t-\tau)\APQ^s}
\big(|u(\tau)|^{\beta-1}u(\tau)\big)\,d\tau
\right\Vert_{L^q}
\leq
C(M+1)^\beta
t^{
1-\frac{n}{2\ell s}\left(\frac1p-\frac1q\right)
-\frac{n(\beta-1)}{2p\ell s}
}.
\end{align}

Multiplying by the corresponding weight in the definition of $Z_T$, we obtain
\[
    \Vert \Psi[u]\Vert_{Z_T}
    \leq
    M
    +
    C(M+1)^\beta
    T^{1-\frac{n(\beta-1)}{2p\ell s}}.
\]
Choosing $T>0$ sufficiently small so that
\[
    C(M+1)^\beta
    T^{1-\frac{n(\beta-1)}{2p\ell s}}
    \leq
    1,
\]
we get $\Psi[u]\in D_{M+1}$.

We next prove that $\Psi$ is a contraction. We use the standard inequality
\[
    \big|
    |a|^{\beta-1}a-|b|^{\beta-1}b
    \big|
    \leq
    C_\beta
    \big(|a|^{\beta-1}+|b|^{\beta-1}\big)|a-b|.
\]
By H\"older's inequality,
\[
\begin{aligned}
\Vert
|u|^{\beta-1}u-|v|^{\beta-1}v
\Vert_{L^p}
\leq
C_\beta
\left(
\Vert u\Vert_{L^{p\beta}}^{\beta-1}
+
\Vert v\Vert_{L^{p\beta}}^{\beta-1}
\right)
\Vert u-v\Vert_{L^{p\beta}}.
\end{aligned}
\]
Therefore, for $u,v\in D_{M+1}$ and $q\in\{p,p\beta\}$,
\[
\begin{aligned}
&\Vert \Psi [u]-\Psi [v]\Vert_{L^q}
\\
&\qquad
\leq
C_\beta
\int_0^t
(t-\tau)^{-\frac{n}{2\ell s}\left(\frac1p-\frac1q\right)}
\left(
\Vert u(\tau)\Vert_{L^{p\beta}}^{\beta-1}
+
\Vert v(\tau)\Vert_{L^{p\beta}}^{\beta-1}
\right)
\Vert u(\tau)-v(\tau)\Vert_{L^{p\beta}}
\,d\tau
\\
&\qquad
\leq
C(M+1)^{\beta-1}
\Vert u-v\Vert_{Z_T}
t^{
1-\frac{n}{2\ell s}\left(\frac1p-\frac1q\right)
-\frac{n(\beta-1)}{2p\ell s}
}.
\end{aligned}
\]
It follows that
\[
    \Vert \Psi[u]-\Psi[v]\Vert_{Z_T}
    \leq
    C(M+1)^{\beta-1}
    T^{1-\frac{n(\beta-1)}{2p\ell s}}
    \Vert u-v\Vert_{Z_T}.
\]
Taking $T>0$ smaller if necessary, we may assume that
\[
    C(M+1)^{\beta-1}
    T^{1-\frac{n(\beta-1)}{2p\ell s}}
    <
    \frac12.
\]
Thus $\Psi$ is a strict contraction on $D_{M+1}$. By Banach's fixed point theorem, there exists a unique fixed point $u\in D_{M+1}$, which is a mild solution of \eqref{eq6.1}.

It remains to show that $u\in C([0,T],L^p(\mathbb R^n)).$
Since
\[
    u(t)
    =
    e^{-t\APQ^s}u_0
    +
    \int_0^t
    e^{-(t-\tau)\APQ^s}
    |u(\tau)|^{\beta-1}u(\tau)\,d\tau,
\]
the integral term tends to zero in $L^p$ as $t\to0^+$ by the estimate already proved in \eqref{lab5.2}. Thus, it is enough to show that
\[
    e^{-t\APQ^s}u_0\to u_0
    \quad
    \text{in } L^p(\mathbb R^n)
    \quad
    \text{as } t\to0^+.
\]
Theorem \ref{Lpmapping} gives uniform boundedness of $e^{-t\APQ^s}$ on $L^p$ for $0<t\leq1$. Hence, by density, it suffices to prove the convergence on a dense subspace of $L^p$. Since finite linear combinations of eigenfunctions of $\APQ$ are dense in $\mathcal S(\mathbb R^n)$, and $\mathcal S(\mathbb R^n)$ is dense in $L^p(\mathbb R^n)$ for $1<p<\infty$, it is enough to consider finite sums of eigenfunctions. For such functions the convergence follows immediately from the spectral representation
\[
    e^{-t\APQ^s}\Phi_j
    =
    e^{-t\lambda_j^s}\Phi_j
    \to
    \Phi_j
    \quad
    \text{as }t\to0^+.
\]
Therefore $u\in C([0,T],L^p(\mathbb R^n))$.

The maximal extension and the blow-up alternative follow from the standard continuation argument: if $ \sup_{0<t<T_{\max}}
    \Vert u(t)\Vert_{L^p}<\infty,$ then the above local existence argument may be restarted at times close to $T_{\max}$, extending the solution beyond $T_{\max}$, which is impossible by maximality. Hence either $T_{\max}=\infty$, or $\lim_{t\to T_{\max}}
    \Vert u(t)\Vert_{L^p}
    =
    \infty.$
This completes the proof.
\end{proof}

We next derive a lower bound for the blow-up rate of finite-time blow-up solutions.

\begin{theorem}
\label{blowuprate}
Let $1<p<\infty$, $s>0$, $\beta>1$, and $u_0\in L^p(\mathbb R^n)$. Assume that
\[
    p>p_c^s:=\frac{n(\beta-1)}{2\ell s}.
\]
Let $u$ be the maximal solution of \eqref{eq6.1} obtained in Theorem \ref{them6.1}. If $T_{\max}<\infty$, then
\[
    \Vert u(t)\Vert_{L^p(\mathbb R^n)}
    \geq
    C
    (T_{\max}-t)^{
    \frac{n}{2p\ell s}
    -
    \frac{1}{\beta-1}
    }
\]
for all $0\leq t<T_{\max}$.
\end{theorem}

\begin{proof} Let $u_0 \in L^p(\mathbb{R}^n)$ and suppose that the corresponding maximal existence time satisfies $T_{\max}<\infty$. Let $u \in C([0,T_{\max}),L^p(\mathbb{R}^n)) $ be the maximal solution of \eqref{eq6.1}, whose existence is guaranteed by Theorem \ref{them6.1}. Fix $\tau \in [0,T_{\max})$ and define \[ v(t)=u(t+\tau), \qquad t \in[0, T_{\max}-\tau ), \] with $v(0)=u(\tau).$ Then, arguing as in the proof of Theorem \ref{them6.1}, we claim that there exists a positive constant $C_1'$ such that 
\begin{equation}\label{eqq6.10} \Vert u(\tau)\Vert_{L^p(\mathbb{R}^n)} + C_1' M^\beta (T_{\max}-\tau)^{1-\frac{n(\beta-1)}{2p\ell s}} > M, \qquad \forall M>0. \end{equation} 

Indeed, assume by contradiction that \eqref{eqq6.10} is false. Then there exists some $M>0$ such that \[ \Vert u(\tau)\Vert_{L^p(\mathbb{R}^n)} + C_1' M^\beta (T_{\max}-\tau )^{1-\frac{n(\beta-1)}{2p\ell s}} \leq M. \] 

By the local existence argument used in Theorem \ref{them6.1}, this estimate implies that the solution $v$ can be extended up to the interval $[0,T_{\max}-\tau].$ Consequently, $u$ would be well-defined at $T_{\max}$, which contradicts the maximality of $T_{\max}$. 

Hence, \eqref{eqq6.10} holds for every fixed $\tau \in[0,T_{\max})$ and for all $M>0$. Now, since \eqref{eqq6.10} is valid for every $M>0$, we choose $M=2\Vert u(\tau)\Vert_{L^p(\mathbb{R}^n)}.$ Substituting this value into \eqref{eqq6.10}, we obtain \[ \Vert u(\tau)\Vert_{L^p(\mathbb{R}^n)} + C_1' \bigl(2\Vert u(\tau)\Vert_{L^p(\mathbb{R}^n)}\bigr)^\beta (T_{\max}-\tau)^{1-\frac{n(\beta-1)}{2p\ell s}} > 2\Vert u(\tau)\Vert_{L^p(\mathbb{R}^n)}. \] 

Therefore, \[ C_1' 2^\beta \Vert u(\tau)\Vert_{L^p(\mathbb{R}^n)}^\beta (T_{\max}-\tau)^{1-\frac{n(\beta-1)}{2p\ell s}} > \Vert u(\tau)\Vert_{L^p(\mathbb{R}^n)}. \] Since $\Vert u(\tau)\Vert_{L^p(\mathbb{R}^n)}>0$, it follows that \[ \Vert u(\tau)\Vert_{L^p(\mathbb{R}^n)}^{\beta-1} \geq C (T_{\max}-\tau)^{-1+\frac{n(\beta-1)}{2p\ell s}}, \] for some positive constant $C$. Hence, \[ \Vert u(\tau)\Vert_{L^p(\mathbb{R}^n)} \geq C (T_{\max}-\tau)^{\frac{n}{2p\ell s}-\frac{1}{\beta-1}}, \] for all $\tau \in[0,T_{\max})$. Replacing $\tau$ by $t$, we obtain \[ \Vert u(t)\Vert_{L^p(\mathbb{R}^n)} \geq C (T_{\max}-t)^{\frac{n}{2p\ell s}-\frac{1}{\beta-1}}, \qquad t\in[0,T_{\max}). \] This completes the proof. \end{proof}

We finally prove a small-data global existence result in the critical exponent case.

\begin{theorem}
\label{criticalglobal}
Let $s>0$, $\beta>1$, and assume that
\[
    p_c^s:=\frac{n(\beta-1)}{2\ell s}>1.
\]
Let $u_0\in L^{p_c^s}(\mathbb R^n)$. If $\Vert u_0\Vert_{L^{p_c^s}(\mathbb R^n)}$
is sufficiently small, then the corresponding solution of \eqref{eq6.1} is global, that is, $T_{\max}=\infty.$
\end{theorem}

\begin{proof}
Choose $r$ such that $p_c^s<r<\beta p_c^s.$
Equivalently, $ \frac{1}{\beta p_c^s}
<\frac1r<
    \frac1{p_c^s}.$
Since $ p_c^s=\frac{n(\beta-1)}{2\ell s},$
this is the same as $ \frac{2\ell s}{n\beta(\beta-1)} <
    \frac{1}{r}   <
    \frac{2\ell s}{n(\beta-1)}.$
Define $\rho :=
    \frac{1}{\beta-1} -
    \frac{n}{2r\ell s}.$
Then $\rho>0$. Moreover, we see that 
\begin{equation}
\label{lessone}
    \rho +1-
    \frac{n(\beta-1)}{2r\ell s}-
    \rho\beta =0.
\end{equation}
We also note that
\[
    \rho\beta<1,
    \qquad
    \frac{n(\beta-1)}{2r\ell s}<1.
\]

Let
\[
    Y
    :=
    \left\{
    u:(0,\infty)\to L^r(\mathbb R^n):
    \Vert u\Vert_Y<\infty
    \right\},
\]
where $\Vert u\Vert_Y
    :=
    \sup_{t>0}
    t^\rho
    \Vert u(t)\Vert_{L^r}.$
    
For $M>0$, set $Y_M
    :=
    \left\{
    u\in Y:
    \Vert u\Vert_Y\leq M
    \right\},$
equipped with the metric $$d(u,v)
    :=
    \sup_{t>0}
    t^\rho
    \Vert u(t)-v(t)\Vert_{L^r}.$$
Then $(Y_M,d)$ is a complete metric space.

Define the fixed-point map
\[
    \mathfrak J_{u_0}(u)(t)
    :=
    e^{-t\APQ^s}u_0
    +
    \int_0^t
    e^{-(t-\tau)\APQ^s}
    \big(|u(\tau)|^{\beta-1}u(\tau)\big)
    \,d\tau.
\]
Assume first that
\begin{equation}
\label{6.11}
    \sup_{t>0}
    t^\rho
    \Vert e^{-t\APQ^s}u_0\Vert_{L^r}
    \leq
    \delta.
\end{equation}
We shall choose $\delta>0$ sufficiently small later.

Let $u,v\in Y_M$. Using Theorem \ref{Lpmapping} with exponents $ \left(\frac r\beta,r\right),$
we obtain
\[
\begin{aligned}
&\Vert
e^{-(t-\tau)\APQ^s}
\big(
|u(\tau)|^{\beta-1}u(\tau)
-
|v(\tau)|^{\beta-1}v(\tau)
\big)
\Vert_{L^r}
\\
&\qquad
\leq
C
(t-\tau)^{-\frac{n(\beta-1)}{2r\ell s}}
\Vert
|u(\tau)|^{\beta-1}u(\tau)
-
|v(\tau)|^{\beta-1}v(\tau)
\Vert_{L^{r/\beta}}.
\end{aligned}
\]
By H\"older's inequality and the pointwise Lipschitz bound for the nonlinearity,
\[
\begin{aligned}
\Vert
|u|^{\beta-1}u
-
|v|^{\beta-1}v
\Vert_{L^{r/\beta}}
\leq
C_\beta
\left(
\Vert u\Vert_{L^r}^{\beta-1}
+
\Vert v\Vert_{L^r}^{\beta-1}
\right)
\Vert u-v\Vert_{L^r}.
\end{aligned}
\]
Since $u,v\in Y_M$, $\Vert u(\tau)\Vert_{L^r} +
    \Vert v(\tau)\Vert_{L^r}
    \leq C M\tau^{-\rho},$
and $\Vert u(\tau)-v(\tau)\Vert_{L^r}
    \leq
    d(u,v)\tau^{-\rho},$ we have

\[
\begin{aligned}
\Vert
e^{-(t-\tau)\APQ^s}
\big(
|u(\tau)|^{\beta-1}u(\tau)
-
|v(\tau)|^{\beta-1}v(\tau)
\big)
\Vert_{L^r}
\leq
C
M^{\beta-1}
d(u,v)
(t-\tau)^{-\frac{n(\beta-1)}{2r\ell s}}
\tau^{-\rho\beta}.
\end{aligned}
\]
Thus,
\[
\begin{aligned}
t^\rho
\Vert
\mathfrak J_{u_0}(u)(t)
-
\mathfrak J_{u_0}(v)(t)
\Vert_{L^r}
\leq
C
t^\rho
M^{\beta-1}
d(u,v)
\int_0^t
(t-\tau)^{-\frac{n(\beta-1)}{2r\ell s}}
\tau^{-\rho\beta}
\,d\tau.
\end{aligned}
\]
Using the change of variables $\tau=t\theta$, we get
\[
\begin{aligned}
t^\rho
\int_0^t
(t-\tau)^{-\frac{n(\beta-1)}{2r\ell s}}
\tau^{-\rho\beta}
\,d\tau
=
t^{
\rho+1
-\frac{n(\beta-1)}{2r\ell s}
-\rho\beta
}
\int_0^1
(1-\theta)^{-\frac{n(\beta-1)}{2r\ell s}}
\theta^{-\rho\beta}
\,d\theta.
\end{aligned}
\]
The integral is finite because $ \rho\beta<1,\,
    \frac{n(\beta-1)}{2r\ell s}<1,$
and the power of $t$ is zero by \eqref{lessone}. Hence there exists $L>0$ such that
\begin{equation}
\label{criticalcontraction}
    d\big(\mathfrak J_{u_0}(u),\mathfrak J_{u_0}(v)\big)
    \leq
    L M^{\beta-1}d(u,v).
\end{equation}
Similarly, taking $v=0$, we obtain
$\Vert \mathfrak J_{u_0}(u)\Vert_Y
    \leq \delta + L M^\beta.$

Choose $M>0$ so small that $L M^{\beta-1}\leq \frac12.$
Then choose $\delta>0$ so small that
\[
    \delta+L M^\beta\leq M.
\]
With these choices, $\mathfrak J_{u_0}$ maps $Y_M$ into itself and is a strict contraction on $Y_M$. Therefore, by Banach's fixed point theorem, there exists a unique global fixed point $u\in Y_M$.

It remains to verify that \eqref{6.11} follows from smallness of $u_0$ in $L^{p_c^s}$. By Theorem \ref{Lpmapping}, for $0<t\leq1$,
\[
    \Vert e^{-t\APQ^s}u_0\Vert_{L^r}
    \leq
    C
    t^{-\frac{n}{2\ell s}
    \left(
    \frac1{p_c^s}-\frac1r
    \right)}
    \Vert u_0\Vert_{L^{p_c^s}}.
\]
Since $\frac{n}{2\ell s}
    \left(
    \frac1{p_c^s}-\frac1r
    \right)  =
    \rho,$
we obtain
\[
    t^\rho
    \Vert e^{-t\APQ^s}u_0\Vert_{L^r}
    \leq
    C
    \Vert u_0\Vert_{L^{p_c^s}},
    \qquad 0<t\leq1.
\]
For $t>1$, the large-time estimate in Theorem \ref{Lpmapping} gives
\[
    \Vert e^{-t\APQ^s}u_0\Vert_{L^r}
    \leq
    C e^{-t\lambda_0^s}
    \Vert u_0\Vert_{L^{p_c^s}},
\]
and hence
\[
    t^\rho
    \Vert e^{-t\APQ^s}u_0\Vert_{L^r}
    \leq
    C
    t^\rho e^{-t\lambda_0^s}
    \Vert u_0\Vert_{L^{p_c^s}}
    \leq
    C
    \Vert u_0\Vert_{L^{p_c^s}}.
\]
Therefore,
\[
    \sup_{t>0}
    t^\rho
    \Vert e^{-t\APQ^s}u_0\Vert_{L^r}
    \leq
    C
    \Vert u_0\Vert_{L^{p_c^s}}.
\]
Thus \eqref{6.11} holds whenever $\Vert u_0\Vert_{L^{p_c^s}}$ is sufficiently small. Consequently, the solution is global. This completes the proof.
\end{proof}

\section{Global well-posedness of  nonlinear fractional heat equations associated with a class of anharmonic oscillators with initial data on Modulation space}\label{Modulation}
In this section, our main aim is to investigate the global well-posedness of the following nonlinear fractional heat equation associated with the anharmonic oscillator $\APQ:$
\begin{equation}
	\begin{cases}
		\partial_t u+ \APQ^s u=\lambda  |u|^{2\beta} u, \\
		u(0, x)=u_0(x),
	\end{cases}
\end{equation}
for $(t,x) \in (0, \infty) \times \mathbb{R}^n,$ where $\beta \in \mathbb{N},$ $\lambda \in \mathbb{C}$ and $s>0.$\\

We begin by recalling some well-known properties of modulation space \cite{feich:mod,than:hs}. The following result is concerned with the algebra property of modulation spaces. 
\begin{lemma} \label{lm}
	Let $m \in \mathbb{N}$ and $p_i,q_j \in [0, 1]$ for $1\leq i \leq m$ such that $\sum_{i=1}^m \frac{1}{p_i}=\frac{1}{p_0}$ and $\sum_{i=1}^m \frac{1}{q_i}=m-1+\frac{1}{q_0}.$ Then, for some $C>0,$ we have
	$$\Bigg\Vert \prod_{i=1}^m f_i \Bigg\Vert_{\M^{p_0,q_0}(\mathbb{R}^n)}\leq C \prod_{i=1}^m \Vert f_i \Vert_{\M^{p_i,q_i}(\mathbb{R}^n)}.$$
\end{lemma}
The next result easily follows from Lemma \ref{lm} and the embedding of modulation spaces (for proof, see e.g. \cite{than:hs}). 
\begin{lemma} \label{multiesti} Let $p,q,r \in [1, \infty]$ and $\frac{1}{r}+2\beta=\frac{2\beta+1}{q}$ for $\beta \in \mathbb{N}.$
	Then the following multi-linear estimate holds:
	$$\Vert |f|^{2\beta} f \Vert_{\M^{p, r}(\mathbb{R}^n)} \leq C \Vert f \Vert_{\M^{p,q}(\mathbb{R}^n)}^{2\beta+1}.$$
\end{lemma}

Now, prove the main result of this section, that is, Theorem \ref{mod_global_nonlinear}, concerning the global well-posedness of nonlinear heat equations associated with the fractional anharmonic oscillator.

\begin{proof}[Proof of Theorem \ref{mod_global_nonlinear}]
We write the equation in its Duhamel form:
\begin{equation}
\label{eq4.5}
    u(t)
    =
    e^{-t\APQ^s}u_0
    +
    \lambda
    \int_0^t
    e^{-(t-\tau)\APQ^s}
    \big(|u(\tau)|^{2\beta}u(\tau)\big)
    \,d\tau .
\end{equation}

We first establish global existence in $X:=L^\infty([0,\infty),\mathcal M^{p,q}).$

By Theorem \ref{mainthmest}, we have the uniform estimate
\begin{equation}
\label{linear_bdd}
    \Vert e^{-t\APQ^s}f\Vert_{\mathcal M^{p,q}}
    \leq
    C\Vert f\Vert_{\mathcal M^{p,q}},
    \qquad t\geq0.
\end{equation}

We next choose an exponent $r$ adapted to the nonlinear estimate. Since $2\beta+1\leq q',$
we can define $r\in[1,\infty]$ by
\begin{equation}
\label{choice_r}
    \frac1r
    =
    \frac{2\beta+1}{q}-2\beta .
\end{equation}
This gives $r\geq q$. Moreover, $\frac1q-\frac1r
    =
    \frac{2\beta}{q'}.$ Hence, by the assumption $\frac{\beta n}{s\ell}<q',$
we obtain
\begin{equation}
\label{sigma_less_one}
    \sigma
    :=
    \frac{n}{2\ell s}
    \left(
    \frac1q-\frac1r
    \right)
    =
    \frac{\beta n}{\ell s q'}
    <1.
\end{equation}

Applying Theorem \ref{mainthmest} with input space $\mathcal M^{p,r}$ and output space $\mathcal M^{p,q}$, we obtain
\begin{equation}
\label{semigroup_kernel_bound}
    \Vert e^{-t\APQ^s}f\Vert_{\mathcal M^{p,q}}
    \leq
    C(t)\Vert f\Vert_{\mathcal M^{p,r}},
\end{equation}
where
\[
    C(t)
    =
    C
    \begin{cases}
        t^{-\sigma}, & 0<t\leq1,\\
        e^{-t\lambda_0^s}, & t\geq1.
    \end{cases}
\]
Since $\sigma<1$, we have $  \int_0^\infty C(t)\,dt<\infty.$

By the multilinear estimate for modulation spaces, see Lemma \ref{multiesti}, and by the choice of $r$ in \eqref{choice_r}, we have
\begin{equation}
\label{nonlinear_mod_est}
    \Vert |u|^{2\beta}u\Vert_{\mathcal M^{p,r}}
    \leq
    C
    \Vert u\Vert_{\mathcal M^{p,q}}^{2\beta+1}.
\end{equation}
Therefore, using Minkowski's inequality, \eqref{semigroup_kernel_bound}, and \eqref{nonlinear_mod_est}, we obtain
\[
\begin{aligned}
&\left\Vert
\int_0^t
e^{-(t-\tau)\APQ^s}
\big(|u(\tau)|^{2\beta}u(\tau)\big)
\,d\tau
\right\Vert_{\mathcal M^{p,q}}
\leq
C
\int_0^t
C(t-\tau)
\Vert u(\tau)\Vert_{\mathcal M^{p,q}}^{2\beta+1}
\,d\tau
\\
&\qquad
\leq
C
\Vert u\Vert_X^{2\beta+1}
\int_0^\infty C(\tau)\,d\tau
\leq
C
\Vert u\Vert_X^{2\beta+1}.
\end{aligned}
\]
Thus, if we define
\[
    \mathfrak T u(t)
    :=
    e^{-t\APQ^s}u_0
    +
    \lambda
    \int_0^t
    e^{-(t-\tau)\APQ^s}
    \big(|u(\tau)|^{2\beta}u(\tau)\big)
    \,d\tau ,
\]
then
\begin{equation}
\label{map_X_bound}
    \Vert \mathfrak T u\Vert_X
    \leq
    C
    \Vert u_0\Vert_{\mathcal M^{p,q}}
    +
    C|\lambda|
    \Vert u\Vert_X^{2\beta+1}.
\end{equation}

Let $B_R
    :=
    \{u\in X:\Vert u\Vert_X\leq R\}.$
Choose $R>0$ sufficiently small so that $ C|\lambda|R^{2\beta}\leq \frac12.$
Then choose $\epsilon>0$ such that $C\epsilon\leq \frac R2.$
If $\Vert u_0\Vert_{\mathcal M^{p,q}}\leq\epsilon$ and $u\in B_R$, then \eqref{map_X_bound} gives
\[
    \Vert \mathfrak T u\Vert_X
    \leq
    \frac R2+\frac R2
    =
    R.
\]
Hence $\mathfrak T$ maps $B_R$ into itself.

We now prove that $\mathfrak T$ is a contraction. We use the pointwise inequality
\[
    \big|
    |z|^{2\beta}z-|w|^{2\beta}w
    \big|
    \leq
    C_\beta
    \big(|z|^{2\beta}+|w|^{2\beta}\big)|z-w|.
\]
The corresponding multilinear estimate gives
\begin{equation}
\label{nonlinear_diff_mod_est}
\begin{aligned}
\Vert
|u|^{2\beta}u-|v|^{2\beta}v
\Vert_{\mathcal M^{p,r}}
\leq
C_\beta
\left(
\Vert u\Vert_{\mathcal M^{p,q}}^{2\beta}
+
\Vert v\Vert_{\mathcal M^{p,q}}^{2\beta}
\right)
\Vert u-v\Vert_{\mathcal M^{p,q}}.
\end{aligned}
\end{equation}
Therefore, for $u,v\in B_R$,
\[
\begin{aligned}
    \Vert \mathfrak T u-\mathfrak T v\Vert_X
    \leq
    C|\lambda|
    \left(
    \Vert u\Vert_X^{2\beta}
    +
    \Vert v\Vert_X^{2\beta}
    \right)
    \Vert u-v\Vert_X
\leq
    C|\lambda|R^{2\beta}
    \Vert u-v\Vert_X.
\end{aligned}
\]
By making $R>0$ smaller if necessary, we may assume that $C|\lambda|R^{2\beta}<1.$
Thus $\mathfrak T$ is a strict contraction on $B_R$ fixed-point. Banach's fixed point theorem gives a unique fixed point $u\in B_R\subset L^\infty([0,\infty),\mathcal M^{p,q}),$
which is the global mild solution of \eqref{nonpro}.

Additionally, assume that $p<\infty$. Our goal is to show that the unique solution $u\in L^{\infty}\bigl([0,\infty),\mathcal{M}^{p,q}(\mathbb{R}^n)\bigr)$
obtained from \eqref{nonpro} is continuous with respect to $t$, that is, $u\in C\bigl([0,\infty),\mathcal{M}^{p,q}(\mathbb{R}^n)\bigr).$
Since the condition $2\beta+1\leq q'$ implies $q<\infty$, it suffices to prove that the semigroup $\{e^{-t\mathcal{A}^s}\}_{t\geq 0}$ is strongly continuous on $\mathcal{M}^{p,q}(\mathbb{R}^n)$. Indeed, once this property is established, the estimates leading to \eqref{linear_bdd}, together with the Banach fixed-point argument employed above, can be repeated in the space $C\bigl([0,\infty),\mathcal{M}^{p,q}(\mathbb{R}^n)\bigr)$
instead of $L^{\infty}\bigl([0,\infty),\mathcal{M}^{p,q}(\mathbb{R}^n)\bigr),$
yielding the desired conclusion.

By \eqref{linear_bdd}, to verify the strong continuity of $e^{-t\mathcal{H}_{P, Q}^s}$ on $\mathcal{M}^{p,q}(\mathbb{R}^n)$, it is enough to show that, for every $f$ belonging to a dense subspace of $\mathcal{M}^{p,q}(\mathbb{R}^n)$, the mapping $ t\longmapsto e^{-t\mathcal{H}_{P, Q}^s}f $
is continuous as an $\mathcal{M}^{p,q}(\mathbb{R}^n)$-valued function.

To this end, we use the embedding \eqref{anhmet012hh} together with an abstract argument from \cite[p.~194]{NicolaRodino}. Since $e^{-t\mathcal{H}_{P, Q}^s}$ is a strongly continuous semigroup on $L^2(\mathbb{R}^n)$ and $\mathcal{A}^{\beta}$ commutes with $e^{-t\mathcal{H}_{P, Q}^s}$ for every $\beta\in\mathbb{N}$, it follows that
$t\longmapsto e^{-t\mathcal{H}_{P, Q}^s}\mathcal{H}_{P, Q}^{\beta}f
=\mathcal{H}_{P, Q}^{\beta}e^{-t\mathcal{H}_{P, Q}^s}f$
is continuous with values in $L^2(\mathbb{R}^n)$ for every $\beta\in\mathbb{N}$.

Moreover, it follows from \cite[p.~194]{NicolaRodino} that the family of seminorms $
p_{\beta}(f):=\Vert \mathcal{H}_{P, Q}^{\beta}f\Vert_{L^2(\mathbb{R}^n)}
$ for $\beta\in\mathbb{N},$
generates the Schwartz topology on $\mathcal{S}(\mathbb{R}^n)$. Consequently, the map $t\longmapsto e^{-t\mathcal{H}_{P, Q}^s}f$
is continuous with values in $\mathcal{S}(\mathbb{R}^n)$. Since $\mathcal{S}(\mathbb{R}^n)$ is continuously embedded into $\mathcal{M}^{p,q}(\mathbb{R}^n)$, the same map is also continuous as an $\mathcal{M}^{p,q}(\mathbb{R}^n)$-valued function. Therefore, $e^{-t\mathcal{H}_{P, Q}^s}$ defines a strongly continuous semigroup on $\mathcal{M}^{p,q}(\mathbb{R}^n)$. As a consequence, the solution $u$ of \eqref{nonpro} belongs to
$ C\bigl([0,\infty),\mathcal{M}^{p,q}(\mathbb{R}^n)\bigr). $

It remains to prove the exponential decay. We work in the space $Y$ defined by \eqref{spaceY}. First, Theorem \ref{mainthmest} gives
\[
    \Vert e^{-t\APQ^s}f\Vert_{\mathcal M^{p,q}}
    \leq
    C e^{-t\lambda_0^s}
    \Vert f\Vert_{\mathcal M^{p,q}},
    \qquad t\geq1.
\]
For $0\leq t\leq1$, the uniform estimate \eqref{linear_bdd} and the boundedness of $e^{t\lambda_0^s}$ imply
\[
    e^{t\lambda_0^s}
    \Vert e^{-t\APQ^s}f\Vert_{\mathcal M^{p,q}}
    \leq
    C\Vert f\Vert_{\mathcal M^{p,q}}.
\]
Hence
\begin{equation}
\label{linear_Y}
    \Vert e^{-t\APQ^s}f\Vert_Y
    \leq
    C
    \Vert f\Vert_{\mathcal M^{p,q}}.
\end{equation}

Now let $u\in Y$. Using \eqref{semigroup_kernel_bound} and \eqref{nonlinear_mod_est}, we estimate the Duhamel term. Set $\Lambda:=\lambda_0^s.$
Then
\[
\begin{aligned}
e^{t\Lambda}
\left\Vert
\int_0^t
e^{-(t-\tau)\APQ^s}
\big(|u(\tau)|^{2\beta}u(\tau)\big)
\,d\tau
\right\Vert_{\mathcal M^{p,q}}
\leq
&C
e^{t\Lambda}
\int_0^t
C(t-\tau)
\Vert u(\tau)\Vert_{\mathcal M^{p,q}}^{2\beta+1}
\,d\tau
\\
&\qquad
\leq
C
\Vert u\Vert_Y^{2\beta+1}
e^{t\Lambda}
\int_0^t
C(t-\tau)
e^{-(2\beta+1)\tau\Lambda}
\,d\tau .
\end{aligned}
\]
We split the integral into two regions.

First, when $0<t-\tau\leq1$, write $\rho=t-\tau$. Then
\[
\begin{aligned}
&e^{t\Lambda}
\int_{t-1}^{t}
(t-\tau)^{-\sigma}
e^{-(2\beta+1)\tau\Lambda}\,d\tau &
=
\int_0^1
\rho^{-\sigma}
e^{-2\beta t\Lambda}
e^{(2\beta+1)\rho\Lambda}
\,d\rho
\leq
C
\int_0^1
\rho^{-\sigma}\,d\rho
<\infty,
\end{aligned}
\]
because $\sigma<1$. Second, when $t-\tau\geq1$, using the exponential part of the semigroup estimate,
\[
\begin{aligned}
e^{t\Lambda}
\int_0^{t-1}
e^{-(t-\tau)\Lambda}
e^{-(2\beta+1)\tau\Lambda}
\,d\tau
=
\int_0^{t-1}
e^{-2\beta \tau\Lambda}
\,d\tau
\leq
\int_0^\infty
e^{-2\beta \tau\Lambda}
\,d\tau
<\infty.
\end{aligned}
\]
Therefore,
\begin{equation}
\label{duhamel_Y}
    \left\Vert
    \int_0^t
    e^{-(t-\tau)\APQ^s}
    \big(|u(\tau)|^{2\beta}u(\tau)\big)
    \,d\tau
    \right\Vert_Y
    \leq
    C
    \Vert u\Vert_Y^{2\beta+1}.
\end{equation}

Similarly, using \eqref{nonlinear_diff_mod_est}, one obtains
\begin{equation}
\label{duhamel_Y_diff}
    \Vert \mathfrak T u-\mathfrak T v\Vert_Y
    \leq
    C|\lambda|
    \left(
    \Vert u\Vert_Y^{2\beta}
    +
    \Vert v\Vert_Y^{2\beta}
    \right)
    \Vert u-v\Vert_Y.
\end{equation}

We now repeat the fixed-point argument in the ball $ B_R^Y:=\{u\in Y:\Vert u\Vert_Y\leq R\}.$
By \eqref{linear_Y} and \eqref{duhamel_Y},
\[
    \Vert \mathfrak T u\Vert_Y
    \leq
    C\Vert u_0\Vert_{\mathcal M^{p,q}}
    +
    C|\lambda|
    \Vert u\Vert_Y^{2\beta+1}.
\]
Choosing $R>0$ sufficiently small and then choosing $\epsilon>0$
sufficiently small so that $C\epsilon\leq \frac R2,
   \,\,
    C|\lambda|R^{2\beta}\leq \frac12,$
we conclude that $\mathfrak T$ maps $B_R^Y$ into itself. Moreover, \eqref{duhamel_Y_diff} shows that $\mathfrak T$ is a strict contraction on $B_R^Y$. Hence the fixed point belongs to $Y$. Consequently, $  \sup_{t\geq0}
    e^{t\lambda_0^s}
    \Vert u(t)\Vert_{\mathcal M^{p,q}}
    <\infty,$
or equivalently,
\[
    \Vert u(t)\Vert_{\mathcal M^{p,q}}
    \leq
    C e^{-t\lambda_0^s},
    \qquad t\geq0.
\]
This proves the exponential decay and completes the proof.
\end{proof}

\section*{Acknowledgement}
 SSM is also supported by the DST-INSPIRE Faculty Fellowship DST/INSPIRE/04/2023/002038.

%
%
%
%
%
%
%
%
%
%
%

\bibliographystyle{amsplain}

\end{document}